\documentclass[12pt]{article}

\usepackage{enumerate}
\usepackage{graphicx}
\usepackage{latexsym,amssymb}
\usepackage{amsthm}
\usepackage{indentfirst}
\usepackage{amsmath}
\usepackage{color}
\usepackage{xcolor}
\usepackage{fourier}
\usepackage[colorlinks=true,backref=page]{hyperref}

\newtheorem{theoremalph}{Theorem}

\newtheorem*{Main Theorem}{Main Theorem}

\newtheorem{Theorem}{Theorem}[section]
\newtheorem*{Theorem A}{Theorem A}
\newtheorem*{Theorem A'}{Theorem A'}
\newtheorem*{Theorem B'}{Theorem B'}
\newtheorem*{Thm}{Theorem}

\newtheorem{Definition}[Theorem]{Definition}
\newtheorem{Proposition}[Theorem]{Proposition}
\newtheorem{Lemma}[Theorem]{Lemma}

\newtheorem{Question}{Question}

\newtheorem{Remark}{Remark}

\newtheorem{Remark-numbered}{Remark}

\newtheorem*{Claim}{Claim}
\newtheorem{Claim-numbered}{Claim}

 \def\NN{{\mathbb N}} 
\def\PP{{\mathbb P}}
 \def\RR{{\mathbb R}} 
\def\TT{{\mathbb T}}

  \def\cG{{\cal G}}  
\def\cB{{\cal B}}   \def\cN{{\cal N}} 
\def\cC{{\cal C}}    \def\cU{{\cal U}}
    
\def\cE{{\cal E}}    
\def\cF{{\cal F}}   \def\cR{{\cal R}} \def\cX{{\cal X}}

\newcommand{\sing}{{\operatorname{Sing}}}

\newcommand{\orb}{\operatorname{Orb}}

\def\dim{\operatorname{dim}}
\def\ind{\operatorname{ind}}

\def\Sing{\operatorname{Sing}}
\def\orb{\operatorname{Orb}}

\def\ud{\operatorname{d}}
\def\e{{\varepsilon}}

\def\det{\operatorname{det}}

\def\wh{\widehat}

\begin{document}

\title{{On the notions of singular domination and (multi-)singular hyperbolicity}}

\author{Sylvain Crovisier, Adriana da Luz,  Dawei Yang and Jinhua Zhang\footnote{S.C was  partially supported by  the ERC project 692925 \emph{NUHGD}.
D.Y  was partially supported by NSFC 11671288, 11822109,  11790274. J. Z was partially supported by the starting grant from Beihang University and by  the ERC project 692925 \emph{NUHGD}.}}


\maketitle

\begin{flushright}
\it To the memory of Shantao Liao.
\end{flushright}

\begin{abstract}	
The properties of uniform hyperbolicity and dominated splitting have been introduced to study the stability of the dynamics of diffeomorphisms.
One meets difficulties when one tries to extend these definitions to vector fields and Shantao Liao has shown that it is more relevant to consider the
linear Poincar\'e flow rather than the tangent flow in order to study the properties of the derivative.

In this paper we define the notion of singular domination, an analog of the dominated splitting for the linear Poincar\'e flow
which is robust under perturbations.
Based on this, we give a new definition of multi-singular hyperbolicity which is equivalent to the one recently introduced by Bonatti-da Luz in~\cite{BdL}.
The novelty of our definition is that  it does not involve the blowup of the singular set and the renormalization cocycle of the linear flows. 
\end{abstract}

\section{Introduction}
The stability and the robust properties of a dynamical system are often associated to invariant structures on the tangent bundle.
For instance the \emph{uniform hyperbolicity}~\cite{A,S} characterizes the structural stability~\cite{M2,H},
but various forms of hyperbolicity have been proposed to investigate other robust properties, such as Liao's star property~\cite{liao-star} or robust transitivity.

One of the weakest hyperbolicity is the notion of \emph{dominated splitting} that appeared in the works of Liao~\cite{liao-domination}, Ma\~n\'e~\cite{M1}, Pliss~\cite{P}.
For diffeomorphisms this occurs once the system is robustly transitive~\cite{BDP}.
In the flow case there exist robustly transitive systems which do not admit any dominated splitting of the tangent flow (see Proposition~\ref{p.robust-transitive-example}),
but for another linear flow, defined by Liao~\cite{liao-poincare}, and called \emph{linear Poincar\'e flow}.

For flows admitting singularities, a direct generalization of the hyperbolicity or domination may not exist, or may not persist under small perturbations (see Proposition~\ref{p.robust-domination-imply-singular domination} below). In dimension 3, \cite{MPP} have defined  \emph{singular hyperbolicity} to make the hyperbolicity of singularities coherent with the hyperbolicity of the periodic orbits, as it occurs inside the Lorenz attractor. In higher dimension, the singular hyperbolicity is not compatible with the coexistence of singularities with different stable dimensions. In order to characterize star systems, such as the $5$-dimensional example~\cite{dL}, a more general property called \emph{multi-singular hyperbolicity}
has been recently introduced by Bonatti and da Luz~\cite{BdL}: the definition (presented in section~\ref{s.equi-definition}) involves in blow up at the singularities,
extended linear Poincar\'e  flow and rescalings by certain dynamical cocyles.

The aim of this text is to give an alternative definition of the multi-singular hyperbolicity without using cocycles and blowup.
To that purpose, we first define and investigate the notion of singular domination for the linear Poincar\'e flow. 
\medskip

Throughout this paper, we consider the set $\cX^1(M)$  of $C^1$-vector fields on a closed manifold $M$.
Given $X\in\cX^1(M)$, we denote by $(\varphi_t)_{t\in\RR}$ the flow generated by $X$ and by $\sing(X)$ the set of singularities.
Its derivative induces a linear flow $(D\varphi_t)$ on the tangent bundle $TM$.
One can also consider the normal bundle $\cN|_{M\setminus\sing(X)}$
obtained as the quotient of the tangent bundle by the flow direction $\RR.X$ on $M\setminus \sing(X)$:
the tangent flow induces the linear Poincar\'e flow $(\Psi_t)$ on the normal bundle (see also section~\ref{ss.poincare}).

The notion of dominated splitting may be defined in the general setting of linear flows:
let us consider a linear bundle $\cB\to \Lambda$ over a space $\Lambda$
and a linear flow $(A_t)$ on $\cB$ which extends a flow $(\varphi_t)$ on $\Lambda$.
An invariant splitting $\cB=\cE\oplus \cF$ into linear subbundles with constant dimension  is dominated for $(A_t)$
if there exist $\eta,T>0$ such that
$$\|A_t|_{\cE(x)}\|\cdot\|A_{-t}|_{\cF(\varphi_t(x))}\|<e^{-\eta t} \quad \textrm{ for any $x\in \Lambda$ and $t>T$}.$$ 
The dimension $i=\dim(\cE)$ is called \emph{index} of the splitting.

The bundle $\cE$ is \emph{uniformly contracted} by $(A_t)$ if there are $\eta,T>0$
such that
$$\|A_{t}|_{\cE(x)}\|<e^{-\eta t} \quad \textrm{ for any $x\in \Lambda$  and $t>T$}.$$
And $\cF$ is \emph{uniformly expanded} if there are $\eta,T>0$ such that
$\|A_{-t}|_{\cF(x)}\|<e^{-\eta t}$ for $x\in \Lambda$, $t>T$.

If the tangent flow has a dominated splitting
$TM|_\Lambda=E^{ss}\oplus F$ with $E^{ss}$ uniformly contracted over an invariant set $\Lambda$,
it is well-known that any point $x\in \Lambda$ admits a well defined stable manifold $W^{ss}(x)$ tangent to $E^{ss}(x)$.
Similarly for a splitting $TM|_\Lambda=E\oplus E^{uu}$ with $E^{uu}$ uniformly expanded,
any point of $\Lambda$ admits an unstable manifold $W^{uu}(x)$ tangent to $E^{uu}(x)$.

\paragraph{a-- Singular domination.}
It is classical that the existence of a dominated splitting for the linear Poincar\'e flow on a non-singular invariant compact  set is a robust property~\cite[Appendix B.1]{BDV}.
But the domination for linear Poincar\'e flow over a set containing singularities may not be preserved after small perturbations.
This motivates the following stronger notion.

 \begin{Definition}\label{d.singular-domination}
Let $X\in \cX^1(M)$ and $\Lambda$ be an invariant compact set.
A $(\Psi_t)$-invariant decomposition $\cN|_{\Lambda\setminus\Sing(X)}=\cN_1\oplus\cN_2$ is a \emph{singular dominated splitting} if 
\begin{enumerate}[i.]
 	\item $\cN_1\oplus\cN_2$ is dominated\footnote{Since $\Lambda\setminus\Sing(X)$ is not compact, one needs to specify a metric on
	$\cN|_{\Lambda\setminus\Sing(X)}$ for which the definition of domination holds: one considers here the quotient metric induced by any Riemannian metric on $M$.};
 	\item  at each singularity $\sigma\in\Lambda\cap\Sing(X),$
 		\begin{itemize}
 		\item either there exists a dominated splitting of the form $T_\sigma M=E^{ss}\oplus F$ for $(D\varphi_t)_{t\in\RR}$ such that $E^{ss}$ is uniformly contracting,
		$\dim(E^{ss})=\dim(\cN_1)$ and $W^{ss}(\sigma)\cap \Lambda=\{\sigma \}$,
 		\item or there exists a dominated splitting of the form $T_\sigma M=E\oplus E^{uu}$ for $(D\varphi_t)_{t\in\RR}$ such that $E^{uu}$ is uniformly expanding,
		$\dim(E^{uu})=\dim(\cN_2)$ and $W^{uu}(\sigma)\cap \Lambda=\{\sigma \}$.
 		\end{itemize}
 \end{enumerate}
\end{Definition}
The links between singular domination of the linear Poincar\'e flow and domination of the tangent flow are discussed in Section~\ref{s.tangent}:
the existence of a dominated splitting for the tangent flow correspond to the special case where one of the bundles is uniformly contracted or expanded
(see Propositions~\ref{p.uniform-bundle} and~\ref{p.domination-criterion}).

Any dominated splitting for a continuous linear cocycle over a compact space is robust under perturbations.
Due to the singularity, this is not always true for arbitrary dominated splitting of the linear Poincar\'e flow.
However the next result shows that the existence of a singular domination is a robust property.
\begin{theoremalph}~\label{p.robust-singular-domination}
Let $X\in \cX^1(M)$ and $\Lambda$ be a compact invariant  set admitting a singular dominated splitting of index $i$. Then there exist neighborhoods $\cU$ of $X$ and $U$ of $\Lambda$ such that for any $Y\in\cU$, the maximal invariant set in $U$ admits a singular dominated splitting of index $i$.
\end{theoremalph}

Conversely, one will show (see Proposition~\ref{p.robust-domination-imply-singular domination} below)
that if a vector field robustly admits a domination in a compact region for the linear Poincar\'e flow, then (under a very mild assumption)
the definition of singular domination holds on that compact region.

\paragraph{b-- Multisingular hyperbolicity.} We also introduce the following notion.

\begin{Definition}~\label{Def:multi-singular-hyperbolic}
Let $X\in \cX^1(M)$. An invariant compact set $\Lambda$ is \emph{multi-singular hyperbolic} if:
\begin{enumerate}[i.]
\item $\Lambda$ admits a singular dominated splitting $\cN^s\oplus \cN^u$;
\item there exist $\eta,T>0$ and a compact  isolating neighborhood $V$ of $\Lambda\cap\Sing(X)$ such that
$$ \|\Psi_t|_{\cN^s(x)}\|<e^{-\eta t} \textrm{ and } \|\Psi_{-t}|_{\cN^u(\varphi_t(x))}\|<e^{-\eta t}
	\text{ whenever $x,\varphi_t(x)\in \Lambda\setminus V$ and $t>T$.}$$
\item each singularity $\sigma\in\Lambda\cap\Sing(X)$ admits a dominated splitting $T_{\sigma} M=E^{ss}\oplus E^{c}\oplus E^{uu}$
with $\dim(E^{ss})=\dim(\cN^s)$, $\dim(E^{uu})=\dim(\cN^u)$ such that if $\rho^{ss}$, $\rho^{uu}$ denote the spectral radii of $D\varphi_1|_{E^{ss}}$ and
$D\varphi_{-1}|_{E^{uu}}$ and if $\rho^c$ is the eigenvalue of $D\varphi_1$ along $ E^c$, then
$$ \max(\rho^{ss},\rho^{uu})<\min(\rho^c,1/\rho^{c})<1.$$
\end{enumerate}
The dimension of $\dim(\cN^s)$ is uniquely defined and is called \emph{index} of $\Lambda$.
\end{Definition}
\begin{Remark}\rm
\begin{enumerate}
\item The multi-singular hyperbolicity we define here is literally different from the notion defined in~\cite{BdL}.
In fact we will see in section~\ref{s.equi-definition} that the two notions coincide (under some very mild assumption), so that we can keep using the same name.
\item Singularities satisfying (iii) are exactly the Lorenz-like singularities (see Section~\ref{s.preliminaries}.e).
In fact, in an invariant compact set satisfying the two first properties of the Definition~\ref{Def:multi-singular-hyperbolic},
the third property holds under a mild condition (see Proposition~\ref{p.lorenz-like-for-non-isolated-sing}).
\item For singularities $\sigma\in\Lambda$ such that $E^c(\sigma)$ is attracting  and   $W^u(\sigma)\cap \Lambda\setminus \{\sigma\}\neq \emptyset$,
the singular domination implies $W^{ss}(\sigma)\cap\Lambda=\{\sigma\}$. An analogous property holds when $E^c(\sigma)$ is expanding.
\end{enumerate} 	
\end{Remark}
 
The multi-singular hyperbolicity is an open property. 
\begin{theoremalph}~\label{thm.robustness-of-multi-singular-hyperbolicity}
Let $X\in \cX^1(M)$ and $\Lambda$ be a multi-singular hyperbolic set. Then there exist a $C^1$-neighborhood $\cU$ of $X$ and a neighborhood $U$ of $\Lambda$ such that the maximal invariant set of $Y\in\cU$ in $U$ is multi-singular hyperbolic.
 \end{theoremalph}
 
One then naturally defines a \emph{multi-singular hyperbolic vector field} as a vector field whose chain-recurrent set is the union of muti-singular hyperbolic sets
and hyperbolic singularities.

From Theorem~\ref{thm.robustness-of-multi-singular-hyperbolicity}, the set of multi-singular vector fields is $C^1$-open. Moreover the definition implies easily that each periodic orbit is hyperbolic
(see Proposition~\ref{p.hyperbolicity-of-erogodic-measure-in-multi-sing}). As a consequence, a multi-singular hyperbolic vector field $X$ has the \emph{star property}, i.e. any vector field in a $C^1$-neighborhood of $X$
has all its periodic orbits and singularities hyperbolic.

The Section~\ref{s.equi-definition} below compares Definition~\ref{Def:multi-singular-hyperbolic}
with the definition of multi-singular hyperbolicity in~\cite{BdL}: in most of the cases they coincide
(see Theorem~\ref{thm.our-imply-bonatti-da-luz} and~\ref{thm.bonatti-da-luz-implies-our}).
It allows us to restate the results from~\cite{BdL} using Definition~\ref{Def:multi-singular-hyperbolic}.

\begin{Thm}[Bonatti-Da Luz]~\label{thm.bonatti-daluz}
The set of multi-singular hyperbolic vector fields is open and dense in the space of star vector fields (for the $C^1$-topology).
 \end{Thm}
 
The following question remains open:
 
 \begin{Question}[\cite{BdL}, Question 1]
Is any star vector field multi-singular hyperbolic?
\end{Question}

\paragraph{c-- Uniform and singular hyperbolicities.} We recall some classical notions.
\begin{Definition}~\label{Def:uniform-hyperbolicity}
Let $X\in \cX^1(M)$. An invariant compact set $\Lambda$ is \emph{uniformly hyperbolic} if:
\begin{enumerate}[i.]
\item $\Lambda$ admits a dominated splitting $TM|_{\Lambda}=E^s\oplus (\RR X) \oplus E^u$ for the tangent flow;
\item $E^s$ is uniformly contracted and $E^u$ is uniformly expanded.
\end{enumerate}
\end{Definition}
\begin{Definition}~\label{Def:singular-hyperbolicity}
Let $X\in \cX^1(M)$. An invariant compact set $\Lambda$ is \emph{singular hyperbolic} if:
\begin{itemize}
\item either $\Lambda$ admits a dominated splitting $TM|_{\Lambda}=E^s\oplus E^{cu}$ for the tangent flow,
$E^s$ is uniformly contracted and $E^{cu}$ is sectionally expanded: there are $\eta,T>0$
such that
$$|\operatorname{Jac}(D\varphi_{-t}|_F)| <e^{-\eta t}  \quad \textrm{ for any $x\in \Lambda$, any $t>T$ and any $2$-plane $F\subset E^{cu}(x)$}.$$
\item or $\Lambda$ admits a dominated splitting $TM|_{\Lambda}=E^{cs}\oplus E^{u}$ such that
$E^{cs}$ is sectionaly contracted and $E^{u}$ is uniformly expanded.
\end{itemize}
\end{Definition}

The multi-singular hyperbolicity generalizes these notions in the following sense
(the first property goes back to~\cite[Proposition 1.1]{D}).

\begin{theoremalph}\label{t.comparison}
Let $X\in\cX^1(M)$ and let $\Lambda$ be an invariant compact set such that   each singularity 
$\sigma\in\Lambda$ is hyperbolic and both $W^s(\sigma)\cap\Lambda\setminus \{\sigma\}$ are $W^u(\sigma)\cap\Lambda\setminus \{\sigma\}$ are non empty. Then:
\begin{enumerate}
\item $\Lambda$ is uniform hyperbolic if and only if $\Lambda$ is multi-singular hyperbolic and does not contain any singularity;
\item $\Lambda$ is singular hyperbolic if and only if $\Lambda$ is multi-singular hyperbolic and all singularities have the same index.
\end{enumerate}
\end{theoremalph}

\paragraph{d-- Several singular dominations.}
We may consider an invariant compact set $\Lambda$ where the linear Poincar\'e flow admits a dominated splitting $\cN=\cN_1\oplus\dots\oplus \cN_\ell$ into more than two bundles:
such that each splitting $(\cN_1\oplus\dots\oplus\cN_k)\oplus(\cN_{k+1}\oplus\dots\oplus\cN_\ell)$ is singular dominated.
As a direct consequence from Definition~\ref{d.singular-domination} we get:

\begin{Remark}
Let $\Lambda$ be a compact set with a singular dominated splitting $\cN=\cN_1\oplus\dots\oplus \cN_\ell$.
Then each hyperbolic singularity $\sigma\in\Lambda$ admits a dominated splitting
$$T_\sigma M=E^s_1\oplus \dots \oplus E^s_k\oplus E^c\oplus E^u_{k+2}\oplus\dots\oplus E^u_\ell,$$
such that
\begin{itemize}
\item each $E^s_i$ is stable and has the same dimension as $\cN_i$,
\item each $E^u_i$ is unstable and has the same dimension as $\cN_i$,
\item $E^c$ has dimension equal to $1+\dim(\cN_{k+1})$,
\item there is a stable manifold $W^{ss}(\sigma)$ tangent to $E^s_1\oplus \dots \oplus E^s_k$ with $W^{ss}(\sigma)\cap\Lambda=\{\sigma\}$,
\item there is an unstable manifold $W^{uu}(\sigma)$ tangent to $E^u_{k+2}\oplus \dots \oplus E^u_\ell$ with $W^{uu}(\sigma)\cap\Lambda=\{\sigma\}$.
\end{itemize}
Moreover the splitting is unique if $W^{s}(\sigma)\cap\Lambda\setminus \{\sigma\}\neq\emptyset$ and $W^{u}(\sigma)\cap\Lambda\setminus \{\sigma\}\neq \emptyset$.
\end{Remark}

\section{Preliminaries}\label{s.preliminaries}
 This section  collects classical notions and properties used in this paper.
\paragraph{a-- Chain recurrence.} Consider a continuous flow $(\varphi_t)_{t\in\RR}$ on a compact metric space $(K,\ud)$. For $\e>0$, a sequence $x_1,\cdots,x_n$ in $K$ is an $\e$-pseudo orbit if for each $1\leq i\leq n-1$, there exists $t_i\geq 1$ such that  $\ud(\varphi_{t_i}(x_i),x_{i+1})<\e$.
One says that $x$ is \emph{chain attainable from $y$}   if for any $\e>0$, there exists an $\e$-pseudo orbit $\{x_i\}_{i=0}^n$ with $x_0=y$, $x_n=x$ and $n\geq 1$.
A set $\Lambda$ is \emph{chain transitive} if for any pair $(x,y)\in\Lambda\times\Lambda$, the first point $x$ is chain attainable from the second point $y$.

A point $x$ is \emph{chain recurrent} if $x$ is chain attainable from itself. The set of chain recurrent points is denoted as $\cR$. 
For $x\in\cR$, we define the \emph{chain recurrence class} of $x$ as the union of the chain transitive sets containing $x$. 
By definition, the chain recurrence classes define a partition of the chain recurrent set into invariant compact sets.
 
\paragraph{b-- The linear Poincar\'e flow and its extension.}\label{ss.poincare}
Given a vector field $X$ on a Riemannian manifold $M$, one defines the normal bundle $\cN$ on the complement
of the singular set $\sing(X)$ in the following way:
 $$\cN|_{M\setminus\sing(X)}=\bigcup_{x\in M\setminus \sing(X)} \big\{v\in T_x M:  <v, X(x)>=0\big\},$$
 where $<\cdot,\cdot>$ denotes the inner product. 
One then defines the \emph{linear Poincar\'e flow}  $(\Psi_t)_{t\in\RR}$ in the following way: for each vector $v\in\cN(x)$ with $x\in M\setminus\sing(X)$ and any $t\in \RR$,
$$\Psi_t(v)=D\varphi_t(v)-\frac{<D\varphi_t(v), X(\varphi_t(x))>}{\|X(\varphi_t(x))\|^2} \cdot X(\varphi_t(x)).$$
 
Following~\cite{LGW}, the linear Poincar\'e flow may be compactified at the singularities as a linear flow $(\wh\Psi_t)_{t\in\RR}$ called \emph{extended linear Poincar\'e flow}.
Consider the projective bundle 
 $$\cG^1=\big\{L\subset T_xM: \textrm{ $x\in M$ and $L$ is a one dimensional linear space in $T_xM $}\big\}.$$
The set $M\setminus\sing(X)$ embeds naturally in $\cG^1$ by the map $x\mapsto \RR X(x)$.
The tangent flow induces a continuous flow $(\wh\varphi_t)_{t\in\RR}$ on $\cG^1$ which extends $(\varphi_t)$: for $L=\RR u$ in $\cG^1$ one defines
$$\wh\varphi_t(\RR u)=\RR D\varphi_t(u).$$
One introduces a normal bundle over $\cG^1(M)$ which extends $\cN|_{M\setminus\sing(X)}$:
for $L\in\cG^1(x)$, $$\cN(L)=\big\{v\in T_x M: \textrm{ $v$ is orthogonal to the linear space $L$}\big \}.$$ 
One defines $(\wh\Psi_t)_{t\in\RR}$ on $\cN$ in the following way: for each $L=\RR u\in \cG^1(x)$ and $v\in\cN(L)$  
$$\wh\Psi_t(v)=D\varphi_t(v)-\frac{<D\varphi_t(v), D\varphi_t(u)>}{\|D\varphi_t(u)\|^2} \cdot D\varphi_t(u).$$
When $x$ is a regular point and $L=\RR X(x)$, then for any $v\in\cN(L)=\cN(x)$ one has $\wh\Psi_t(v)=\Psi_t(v)$.

\paragraph{c-- Lyapunov exponents.} Consider $X\in \cX^1(M)$ and an invariant probability measure $\mu$. The measure is \emph{regular} if $\mu(\sing(X))=0$.

We recall Oseledec theorem. 
For $\mu$-almost every $x\in M$, there are $k=k(x)$ numbers $\lambda_1(x)<\lambda_2(x)<\cdots<\lambda_k(x)$ and a splitting 
$T_xM=E_1(x)\oplus E_2(x)\oplus\cdots\oplus E_k(x)$ such that for any $1\le i\le k$ and any unit vector $v\in E_i$,
		$$\lim_{t\to\pm\infty}\frac{1}{t}\log\|D\varphi_t(v)\|=\lambda_i.$$
If $\mu$ is ergodic, then $k$ and $\lambda_1,\cdots,\lambda_k$ are constants on a full $\mu$-measure set.

A regular ergodic measure is \emph{hyperbolic} if  it has only one vanishing Lyapunov exponent (which is given by the flow direction). 
Equivalently, there exists a measurable splitting of the normal bundle $\cN=\cN_1\oplus\dots\oplus\cN_\ell$ defined on a set with full $\mu$-measure
which is invariant under the linear Poincar\'e flow and non-zero numbers $\lambda'_1(x)<\cdots<\lambda'_\ell(x)$ such that for $\mu$-almost every point $x$,
any $1\le j\le \ell$ and any unit vector $v\in \cN_j$, the quantity
$\frac{1}{t}\log\|\Psi_t(v)\|$ converges to $\lambda'_j$ as $t\to\pm\infty$.
The numbers $\lambda'_j$ coincide with the non-zero Lyapunov exponents $\lambda_i$ of $\mu$.

\paragraph{d-- Dynamics above hyperbolic singularities.}
The \emph{index} $\ind(\sigma)$ of a hyperbolic singularity is the dimension of its stable space.
The following fundamental result allows to exile the strong stable manifold of a singularity from a compact invariant set containing the singularity.
It comes from~\cite{LGW}. 
\begin{Proposition}~\label{p.strong-stable-outside}
Consider $X\in\cX^1(M)$, a hyperbolic singularity $\sigma$
and a $(\widehat \varphi_t)$-invariant compact set $\wh\Lambda$ in the projective tangent space $\cG^1(\sigma)$.
If $\wh\Lambda$ admits a dominated splitting $\wh\cN|_{\wh\Lambda}=\cN_1\oplus \cN_2$ for $(\wh\Psi_t)$ of index $i<\ind(\sigma)$ and
if $\wh\Lambda$ intersects the projective space of $E^u(\sigma)$, then 
	\begin{itemize}
		\item $E^s(\sigma)$ has a finer dominated splitting $E^{ss}(\sigma)\oplus E^{cs}(\sigma)$ for $(D\varphi_t)_{t\in\RR}$ with $\dim(E^{ss}(\sigma))=i$;
		\item any line $L\subset E^{ss}(\sigma)\oplus E^u(\sigma)$
		which is not contained in $E^{ss}(\sigma)\cup E^u(\sigma)$ is disjoint from $\wh\Lambda.$
	\end{itemize}
\end{Proposition}
\proof
Up to changing the metric, one can assume that the bundles in the hyperbolic splitting $E^s(\sigma)\oplus E^u(\sigma)$ are orthogonal to each other.
In particular over points of $\widehat \Lambda$ contained in the projective space of $E^{s}$, the bundle $\cN_2$ contains $E^u(\sigma)$

Consider $L\in\wh\Lambda$ which is contained in the projective space of $E^u(\sigma)$. As $E^s(\sigma)$ is orthogonal to the 1-dimensional linear space given by $L$, one has $\wh\Psi_t(L)|_{ E^s(\sigma)}=D\varphi_t|_{ E^s(\sigma)}$. By the domination $\cN_1\oplus \cN_2$, the space $E^s(\sigma)$ splits into two dominated sub-bundles $E^{ss}=\cN_1$ and $E^{cs}$ over the orbit of $L$.
Since the cocycle $\wh\Psi_t(L)|_{ E^s(\sigma)}$ is constant over the orbit of $L$, these bundles are constant as well.
This implies that $E^s(\sigma)$ admits a dominated decomposition $E^{ss}\oplus E^{cs}$ for $(D\varphi_t)$.

Up to changing the metric, we will further assume that $E^{ss}(\sigma)$ is orthogonal to $E^{cs}(\sigma).$ 

In order to prove the second property, one will suppose by contradiction that there exist a non-vanishing vector $v^{ss}\in E^{ss}(\sigma)$ and a non-vanishing vector $v^u\in E^u(\sigma)$ such that the line $L_0=\RR(v^{ss}+v^u)$ belongs to $\wh\Lambda$.
 Then $v^{ss}\cdot \|v^u\|^2-v^u\cdot \|v^{ss}\|^2$ belongs to  
 $\cN_{L_0}$. Now, we look at the orbit of $v^{ss}\cdot \|v^u\|^2-v^u\cdot \|v^{ss}\|^2\in \cN_{L_0}$ under the extended linear Poincar\'e flow.
 By definition,
 \begin{eqnarray}\label{e.lpf}
&& \wh\Psi_t\big(v^{ss}\cdot \|v^u\|^2-v^u\cdot \|v^{ss}\|^2\big)\nonumber\\
&&\quad\quad =D\varphi_t(v^{ss})\|v^u\|^2-D\varphi_t(v^u)\|v^{ss}\|^2\nonumber\\
&&\quad\quad\quad\quad\quad\quad\quad\quad -\frac{\|D\varphi_t(v^{ss})\|^2\|v^u\|^2-\|D\varphi_t(v^u)\|^2\|v^{ss}\|^2}{\|D\varphi_t(v^{ss})\|^2+\|D\varphi_t(v^u)\|^2}\cdot \big(D\varphi_t(v^{ss})+D\varphi_t(v^u)\big)\nonumber\\
&&\quad\quad =\frac{\|v^{ss}\|^2+\|v^{u}\|^2}{\|D\varphi_t(v^{ss})\|^2+\|D\varphi_t(v^u)\|^2}\cdot \big(\|D\varphi_t(v^u)\|^2 D\varphi_t(v^{ss})- \|D\varphi_t(v^{ss})\|^2 D\varphi_t(v^u)\big).
 \end{eqnarray}

Let us first assume that $v^{ss}\cdot \|v^u\|^2-v^u\cdot \|v^{ss}\|^2$ belongs to $\cN_1|_{L_0}$.
As $\wh\varphi_t(L_0)$ accumulates on a subset  $\alpha$ of the projective space of $E^{ss}(\sigma)$ when $t$ tends to $-\infty$, the invariance of the bundle $\cN_1$ implies that  $\RR\cdot\widehat \Psi_t(v^{ss}\cdot \|v^u\|^2-v^u\cdot \|v^{ss}\|^2)$ accumulates in $\cN_1|_{\alpha}$;  however from~\eqref{e.lpf},
the lines $\RR\cdot\widehat \Psi_t(v^{ss}\cdot \|v^u\|^2-v^u\cdot \|v^{ss}\|^2)$ accumulate on a linear subspace in $E^u(\sigma)$ when $t$ tends to $-\infty$. This is a contradiction since $\cN_2|_{\alpha}$ contains $E^u(\sigma)$
and intersects $\cN_1|_{\alpha}$ trivially.

We are now reduced to the case where $v^{ss}\cdot \|v^u\|^2-v^u\cdot \|v^{ss}\|^2$ does not belong to $\cN_1|_{L_0}$. When $t$ tends to $+\infty$, the lines $\wh\varphi_t(L_0)$ accumulate on a subset $\omega$ of the projective space of $E^{u}(\sigma)$ and (from the domination) $\RR\cdot\Psi_t(v^{ss}\cdot \|v^u\|^2-v^u\cdot \|v^{ss}\|^2)$ accumulates inside a linear space in $\cN_2|_{\omega}$; however from~\eqref{e.lpf}, the lines $\RR\cdot\Psi_t(v^{ss}\cdot \|v^u\|^2-v^u\cdot \|v^{ss}\|^2)$ accumulate  in the linear space $E^{ss}(\sigma)$ which is a contradiction. This proves the second item.
\endproof

\paragraph{e-- Lorenz-like singularities.}\label{ss.lorenz-like} A hyperbolic singularity is Lorenz-like if its hyperbolic splitting $T_\sigma M=E^s\oplus E^u$,
its smallest positive Lyapunov exponent $\lambda^u$ and
its largest negative Lyapunov exponent $\lambda^s$ satisfy one of the following properties:
\begin{itemize}
\item either there exists a dominated splitting $E^s=E^{ss}\oplus E^c$ with $\dim(E^c)=1$ and $\lambda^s+\lambda^u>0$,
\item or  there exists a dominated splitting $E^u=E^{c}\oplus E^{uu}$ with $\dim(E^c)=1$ and $\lambda^s+\lambda^u<0$.
\end{itemize}
Note that this is equivalent to the property stated in the third item of Definition~\ref{Def:multi-singular-hyperbolic}.

\paragraph{f-- Star vector fields.}
A vector field $X\in\cX^1(M)$ is \emph{star} if
for any vector field $Y$ in a $C^1$-neighbor\-hood of $X$, all the periodic orbits and singularities of $Y$ are hyperbolic.

\begin{Theorem}[Liao~\cite{liao-star}]\label{thm:fundamental-property}
For any star vector field $X\in\cX^1(M)$,  there exist $\eta,T>0$ and a $C^1$ neighborhood $\cU$ of $X$ with the following properties. For any $Y\in\cU$ and  any periodic orbit $\gamma$ of $Y$ with period $\pi(\gamma)$ larger than $T$, let us denote $\cN_{\gamma}=\cN^s\oplus\cN^u$ the hyperbolic splitting of the linear Poincar\'e flow $(\Psi_t^Y)_{t\in\mathbb{R}}$ associated to $Y$. Then for each $p\in\gamma$, one has 
	$$\|\Psi^Y_t|_{\cN^s(p)}\|\cdot\|\Psi^Y_{-t}|_{\cN^u(\varphi_t(p))}\|<e^{-2t\eta}\textrm{ for each $t\geq T$};$$
	$$\prod_{i=0}^{[\pi(\gamma)/T]-1}\|\Psi^Y_T|_{\cN^s(\varphi^Y_{iT}(p))}\|\leq e^{-\eta\pi(\gamma)} \text{ and }
	\prod_{i=0}^{[\pi(\gamma)/T]-1}\|\Psi^Y_{-T}|_{\cN^u(\varphi^Y_{iT}(p))}\|\leq e^{-\eta\pi(\gamma)}.$$
\end{Theorem}

\paragraph{g-- Connecting lemma.} This flow version of the connecting lemma comes from~\cite{WX,W}.
\begin{Theorem}~\label{thm.connecting-lemma}
Let $X\in \cX^1(M)$. For any $C^1$-neighborhood $\cU$ of $X$, there exist $T>0$, $\rho\in(0,1)$ and $d_0>0$ such that for any point $x\in M$ which is non-periodic and non-singular under the flow $(\varphi_t^X)_{t\in\RR}$ generated by $X$, one has the following property.

For any $d\in(0,d_0)$, and any points $p,q\notin \Delta_T(x,d):=\cup_{t\in[1,T]} \varphi_t^X(B_d(x))$, if the forward orbit of $p$ and the backward orbit of $q$ intersect $B_{\rho\cdot d}(x)$, then there exists $Y\in\cU$ such that $q$ is on the forward orbit of $p$ under the flow $(\varphi_t^Y)_{t\in\RR}$. Moreover, $Y(z)=X(z)$ for $z\in M\setminus\Delta_T(x,d)$.
\end{Theorem}

 \section{Singular domination}~\label{s.singular-domination}
In this section we discuss the notion of singular domination introduced in the introduction and prove that it is a robust property (Theorem~\ref{p.robust-singular-domination}).
We also build a robust example of a flow with no dominated splitting of the tangent bundle (section~\ref{s.tangent})
which shows that the linear Poincar\'e flow is more adapted than the tangent flow for studying the dynamics of vector fields.
Finally, we motivate the definition of singular domination by proving that a robust dominated splitting of the linear Poincar\'e flow
satisfies Definition~\ref{d.singular-domination}.

\subsection{Dominated splitting of the tangent flow}\label{s.tangent}
We first discuss the domination for the tangent flow:
the next statement (an improved version of~\cite[Theorem B]{BGY}) shows that it constraints the tangent behavior.
\begin{Proposition}\label{p.uniform-bundle}
Let $X\in \cX^1(M)$ and $\Lambda$ be a chain-transitive invariant compact set such that
\begin{itemize}
	\item all the singularities in $\Lambda$ are hyperbolic;
	\item $\Lambda$ admits a dominated splitting  $T_\Lambda M=E\oplus F$ for the tangent flow.
\end{itemize}
Then either $E$ is uniformly contracted, or $F$ is uniformly expanded.
\end{Proposition}

Let us recall the notion of cone field. Given a continuous splitting  $TM|_{\Lambda}=E\oplus F$ over a compact set $\Lambda$, one defines at each point $x\in \Lambda$ a cone field around $F$ of angle $\alpha>0$   as  
$$\cC^F_\alpha(x)=\big\{v\in T_xM: v=v^E+v^F, v^E\in E, v^F\in F \textrm{ and } \|v^E\|\leq \alpha\|v^F\|\big\}.$$

\begin{proof}[Sketch of the proof of Proposition~\ref{p.uniform-bundle}]
Let $x_0$ be a regular point in $\Lambda$. Without loss of generality, one assume $X(x_0)\notin E(x_0)$.
Hence there exists $\alpha>0$ such that $X(x_0)\in \cC^F_{\alpha}(x_0).$
One extends the cone field $\cC^F_\alpha$ continuously to a neighborhood $U$ of $\Lambda$. By domination, up to shrinking $U$, there exist $T>0$ and $\lambda\in(0,1)$ such that $D\varphi_t(\cC^F_{\alpha}(y))\subset \cC_{\lambda^2\alpha}^F(\varphi_t(y))$ for any $t\geq T$ and any $y\in \cup_{s\in[0,t]}\varphi_{-s}(U)$.
Note that in the definition of chain-recurrence, there is no loss of generality if one only considers pseudo orbits,
whose times $t_i$ between the jumps are larger or equal to $T$ (see Section~\ref{s.preliminaries}.a).

\begin{Lemma}
 	For any regular point $y\in\Lambda\setminus\sing(X)$, one has $X(y)\in\cC^F_{\alpha}(y).$
 \end{Lemma} 	
\proof
Fix $y_0\in\Lambda\setminus\sing(X)$ and a small neighborhood $V\subset U$ of $\sing(X)\cap\Lambda$ such that $x_0,y_0\notin V$  and $\bigcup_{s\in [-2T,2T]}\varphi_s(V)\subset U$.
As all the singularities in $\Lambda$ are hyperbolic, for each singularity $\sigma\in\Lambda$, one fixes fundamental domains $\Delta^s(\sigma)\subset V$ and $\Delta^u(\sigma)\subset V$ of the stable and unstable manifolds of $\sigma$ respectively. For $\e_0>0$ small, let us denote $\Delta_{\e_0}^s(\sigma),\Delta_{\e_0}^u(\sigma)$  the $\e_0$-neighborhoods of $\Delta^s(\sigma),\Delta^u(\sigma)$ such that
\begin{itemize}
	\item[(a)] the sets $\{\Delta_{\e_0}^s(\sigma),\Delta_{\e_0}^u(\sigma)\}_{\sigma\in\Sing(X)\cap\Lambda}$ are pairwise disjoint and disjoint from $\Sing(X)$,
	\item[(b)] for any $x_1,x_2$ in a same $\Delta_{\e_0}^*(\sigma)$, if $X(x_1)\in\cC^F_{\lambda^2\alpha}(x_1)$ and $\ud(x_1,x_2)<4\e_0$ then $X(x_2)\in\cC^F_{\lambda\alpha}(x_2)$,
\end{itemize}
By the Inclination lemma, the following property holds:
\begin{itemize}
	\item[(c)]
for any  $\sigma\in\Lambda\cap\sing(X)$ and any points $x_1\in \Delta_{\e_0}^s(\sigma)$, $x_2\in\Delta^u_{\e_0}(\sigma)$, there exist $\wh x_1 \in\Delta_{\e_0}^s(\sigma)$, $\wh x_2\in \Delta_{\e_0}^u(\sigma)$ and $t>T$ such that $\ud (x_i,\wh x_i)<2\e_0$ and $\wh x_2=\varphi_t(\wh x_1)$.
\end{itemize}
As each singularity $\sigma\in\Lambda$ is hyperbolic, there exist $\e_1 $ small and a neighborhood $W_{\sigma}$ such that
\begin{itemize}
	\item[(d)] for any $x\in W_{\sigma}$ and any $\e_1$-pseudo orbit $\{z_i\}_{i=0}^k$,
\begin{itemize}
	\item if  $z_0=x$, $z_k\notin V$, there exist $0\leq i_0< k$ and $0\leq t\leq t_{i_0}$ such that $\varphi_t(z_{i_0})\in \Delta^u_{\e_0/2}(\sigma)$;
	\item  if $z_0\notin V$, $z_k=x$, there exist   $0\leq i_0< k$ and $0\leq t\leq t_{i_0}$ such that $\varphi_t(z_{i_0})\in \Delta^s_{\e_0/2}(\sigma)$.
\end{itemize} 
\end{itemize}
Let $W:=\cup_{\sigma\in\sing(X)\cap\Lambda} W_{\sigma}$. By continuity of $\cC^F_\alpha$, $\mathbb{R} X$ and $(\varphi_t)$, 
there exist $\e_2,\e_3>0$ such that
\begin{itemize}
	\item[(e)] for any $x,z\in U\setminus W$, if $\ud(x,z)<\e_2$ and $X(x)\in\cC^F_{\lambda\alpha}(x)$, then $X(z)\in\cC^F_\alpha(z),$ 
\item[(f)]  for any $x,y\in M $, if $\ud(x,z)<\e_3$, then $\ud(\varphi_s(x),\varphi_s(z))<\e_2$ for any $|s|\leq T$.
\end{itemize}

As $\Lambda$ is chain transitive, for $\e<\frac{1}{4}\min\{\e_0,\e_1,\e_2,\e_3 \}$ there exists an $\e$-pseudo orbit $\{z_i\}_{i=0}^l$ connecting $x_0$ to $y_0$,
with time $\{t_i\}_{i=0}^{l-1}$ larger than $2T$.
Let $I\subset\{0,\cdots,l \}$ be the set of the possible integers  $i$  such that $z_i\in W$. By (c) and (d), for each $i\in I$, there exist $0\leq i^-<i\leq i^+<l$ and times $0\leq t^-\leq t_{i^-}$,  $0\leq t^+\leq t_{i^+}$ such that $\varphi_{t^-}(z_{i^-})\subset \Delta_{\e_0}^s(\sigma)$ and $\varphi_{t^+}(z_{i^+})\in\Delta^u_{\e_0}(\sigma)$ for some singularity $\sigma$.
Then, there exist $y_{i}\in \Delta^s_{\e_0}(\sigma)$ and $\tau_i>T$ such that 
$$\max\big\{\ud(y_i, \varphi_{t^-}(z_{i^-})),\ud(\varphi_{\tau_i}(y_i),\varphi_{t^+}(z_{i^+}))\big\}<2\e_0 \textrm{ and } \varphi_{\tau_i}(y_i)\in\Delta^u_{\e_0}(\sigma).$$
Now, one replaces the pseudo orbit segment from $\varphi_{t^-}(z_{i^-})$ to $\varphi_{t^+}(z_{i^+})$ by the true orbit segment $\{\varphi_t(y_i )\}_{t\in[0,\tau_i]}$, for each $i\in I$.
If $t^-<T$, one replaces the pseudo orbit segment between
$\varphi_{-T+t^-}\circ\varphi_{t_{i^--1}}(z_{i^--1})$ and $\varphi_{t_{i^--1}}(z_{i^--1})$ by the orbit segment
between $\varphi_{-T+t^-}(z_{i^-})$ and $z_{i^-}$.

By the choice of $V$ and (f), one obtains in this way a new pseudo orbit connecting $x_0$ to $y_0$ such that 
\begin{itemize}
	\item the times between the jumps are larger than $T$;
	\item all the jumps avoid $W$;
	\item  each jump avoiding the sets $\Delta^*_{\e_0}(\sigma)$  has size smaller than $\e_2$;
	\item  each jump in $\Delta^*_{\e_0}(\sigma)$ has size smaller than $4\e_0$.
\end{itemize}
By (b), (e), the contraction of the cone field $\cC_\alpha^F$ gives $Y(y)\in\cC^F_{\alpha}(y).$
\qed

The contraction of the cone field $\cC^F_\alpha$ implies that $X(y)\in F(y)$ for any $y\in \Lambda$.
Then using the domination one concludes as in~\cite{BGY}  that $E$ is uniformly contracted.
\end{proof}

As a consequence, one builds robustly transitive flows whose tangent bundle does not admit any domination, which contrasts with~\cite{BDP}.
This shows that the dominated splittings should be rather   searched for the linear Poincar\'e flow than the tangent flow.

\begin{Proposition}\label{p.robust-transitive-example}
There exists an open set of $C^1$-vector fields with no singularity on a manifold of dimension $5$:
\begin{itemize}
\item whose dynamics is robustly transitive,
\item whose tangent flow does not admit any dominated splitting.
\end{itemize}
\end{Proposition}
\begin{proof}
\cite{BV} builds a robustly transitive diffeomorphism $f$ (i.e. each nearby $C^1$-diffeomorphism admits a dense orbit) on $\TT^4$ such that 
\begin{itemize}
	\item $f$ admits a dominated splitting of the form $T\TT^4=E\oplus F$ where $\dim(E)=\dim(F)=2.$
	\item  $E$ is neither uniformly contracting nor uniformly expanding, and so is $F$. 
	\item  Neither $E$ nor $F$ can be split into non-trivial dominated sub-bundles.
\end{itemize}
One considers the suspension of $f$ and one gets a $C^1$-vector field $X_f$ which is also robustly transitive and has no-singularities. If the tangent flow of $X_f$ admits a dominated splitting $E_1\oplus E_2$, then the robust transitivity and Proposition~\ref{p.uniform-bundle} imply that $E_1$ or $E_2$ is hyperbolic, which in return implies that for $f$  the bundle $E$ or $F$ is uniformly hyperbolic. In summary, the tangent flow of $X_f$ does not admit domination.
\end{proof}

Proposition~\ref{p.uniform-bundle} shows that a dominated splitting of the tangent flows implies a dominated splitting of the linear Poincar\'e flow.
The next proposition (due to~\cite[Lemma 2.13]{Gan-Yang}) gives a criterion to obtain a dominated splitting on the tangent flow when the linear Poincar\'e flow is dominated.

\begin{Proposition}\label{p.domination-criterion}
Let $\Lambda$ be an invariant compact set for $X\in\cX^1(M)$ with a dominated splitting $\cN|_{\Lambda\setminus\Sing(X)}=\cN_1\oplus\cN_2$
of index $i$ for the linear Poincar\'e flow. Assume furthermore that:
\begin{itemize}
\item $\Lambda\setminus\Sing(X)$ is dense in $\Lambda$,
\item there exist  $\eta,T>0$ such that
$\|\Psi_t|_{\cN_1}\|\leq e^{-\eta t}\frac{\|X(\varphi_t(x))\|}{\|X(x)\|}$ for all $x\in\Lambda$ and $t>T$.
\end{itemize}
Then the tangent flow over $\Lambda$ admits a dominated splitting $TM|_\Lambda=E^s\oplus F$ of index $i$
and $E^s$ is uniformly contracted.
\end{Proposition}

\subsection{Robustness of the singular domination: proof of Theorem~\ref{p.robust-singular-domination}}
Due to the lack of compactness of the linear Poincar\'e flow over $M\setminus \Sing(X)$, the robustness of the singular domination
is not a direct consequence of the robustness of dominated splittings for continuous linear cocycles over compact spaces.
We now state and prove a more precise version of Theorem~\ref{p.robust-singular-domination}.
\begin{Theorem}\label{t.Aprime}
Let $X\in\cX^1(M)$, $\eta,T>0$ and $\Lambda$ be a compact invariant set admitting a singular dominated splitting $\cN|_{\Lambda\setminus\Sing(X)}=\cN_1\oplus\cN_2$ of index $i$
which is $(\eta, T)$-dominated.

Then there exist  a $C^1$-neighborhood $\cU$ of $X$ and a neighborhood $U$ of $\Lambda$ such that for each $Y\in\cU$, the maximal invariant set $\Lambda_Y$ of $Y$ in $U$ admits a 
singular dominated spliting of index $i$ which is $(\eta, T)$-dominated. 
\end{Theorem}

Before proving the robustness of a singular domination, we need some preparation.
\begin{Lemma}~\label{l.robust-escaping}
	Let $X\in\cX^1(M)$ and let $K\subset \Lambda$ be  two invariant compact sets. Assume that  
	\begin{itemize}
		\item $K$ admits a partially hyperbolic splitting $T_K M=E^{ss}\oplus F$ for the tangent flow of index $i$;
		\item  for any $x\in K$, one has $W^{ss}(x)\cap \Lambda=\{x\}$.
	\end{itemize} 
Then there exist   neighborhoods $\cU$ of $X$, $V$ of $K$ and $U$ of $\Lambda$ such that for any $Y\in \cU$,
\begin{itemize}
	\item the maximal invariant set $K_Y$ in $V$ admits a partially hyperbolic splitting of index $i$;
	\item the maximal invariant set $\Lambda_Y$ in $U$ satisfies $W^{ss}(x)\cap\Lambda_Y=\{x \}$ for any $x\in K_Y.$
\end{itemize}
 
\end{Lemma}

\proof
Since the partial hyperbolicity is robust, the first item holds for a $C^1$-neighborhood $\cU$ of $X$ and a neighborhood $W$ of $K$.
Let us assume by contradiction that the second item does not hold and that there exist sequence of vector fields $Y_n\in\cU$,
neighborhoods $V_n\subset W$ of $K$, neighborhoods $U_n$ of $\Lambda$ and points $x_n$ such that:
\begin{itemize}
	\item $Y_n$ tends to $X$ in $C^1$-topology, $\cap_{n\in\NN} V_n=K$ and $\cap_{n\in\NN} U_n=\Lambda;$
	\item maximal invariant sets $K_n,\Lambda_n$ of $Y_n$ in $V_n,U_n$ satisfy $x_n\in K_n$, $(W^{ss}(x_n)\setminus\{x_n\})\cap \Lambda_n\neq\emptyset.$
\end{itemize}
Each point $x$ in $K_n$ has local strong stable manifolds which vary continuously in the $C^1$-topology with $x$ and $Y_n$.
Up to replacing $x_n$ by iterates, one can consider constants $\e_0>\e_1>0$ such that
$W^{ss}_{\e_0}(x_n)\setminus W^{ss}_{\e_1}(x_n)\cap\Lambda_n\neq\emptyset.$
Up to taking a subsequence, $(x_n)$ converges to a point $x\in K$ and  
$W^{ss}_{\e_0}(x)\setminus W^{ss}_{\e_1}(x)\cap\Lambda\neq\emptyset$ which gives the contradiction.
\endproof

\proof[Proof of Theorem~\ref{t.Aprime}]
We first address the singularities. Let us denote by $S_-$ the set of singularities $\sigma\in \Lambda$
admitting a splitting $T_\sigma M=E^{ss}\oplus F$ with $\dim(E^{ss})=i$ and $W^{ss}(\sigma)\cap \Lambda=\{\sigma \}$.
Analogously, one defines $S_+$ the set of singularities $\sigma\in \Lambda$
admitting a splitting $T_\sigma M=E\oplus E^{uu}$ with $\dim(E^{uu})=\dim(M)-i-1$ and $W^{uu}(\sigma)\cap \Lambda=\{\sigma \}$.

Since local strong stable and unstable manifolds vary continuously with respect to the points,
each $\sigma\in S_-$ (resp. $\sigma\in S_+$) admits a neighborhood $U_\sigma$ with $U_\sigma\cap\sing(X)\cap\Lambda\subset S_-$ (resp. $\subset S_+$).
As  $\sing(X)\cap\Lambda$ is  compact, there exist two open sets $V^-,V^+$ such that:
\begin{itemize}
	\item  $\sing(X)\cap\Lambda\subset V^-\cup V^+$;
	\item $\overline{V^-}\cap\sing(X)\cap\Lambda\subset S_-$ and $\overline{V^+}\cap\sing(X)\cap\Lambda\subset S_+$. 
\end{itemize} 
Applying Lemma~\ref{l.robust-escaping} to $\overline{V^-}\cap\sing(X)\cap\Lambda$ (resp. $ \overline{V^+}\cap\sing(X)\cap\Lambda$), one gets a neighborhood $\cU_0$ of $X$, a neighborhood $U_0$ of $\Lambda$ and an open subset $V$ of $V^-\cup V^+$ such that for each $Y\in\cU_0$,
\begin{itemize}
	\item $\sing(Y)\cap U_0\subset V$;
	\item any singularity $\sigma$ in $\sing(Y)\cap V^-$ (resp. $\sing(Y)\cap V^+$) admits a splitting $E^{ss}\oplus F$ (resp. $E\oplus E^{uu}$)
with $\dim(E^{ss})=i$ (resp. $\dim(E^{uu})=\dim(M)-i-1$);
	\item the maximal invariant set $\Lambda^0_Y$ of $Y$ in $U_0$ satisfies
$W^{ss}(\sigma)\cap\Lambda^0_Y\subset \{\sigma\}$ for each $\sigma$ in $\sing(Y)\cap V^-$ and $W^{uu}(\sigma)\cap\Lambda^0_Y\subset \{\sigma\}$ for each $\sigma$ in $\sing(Y)\cap V^+$.
\end{itemize}
This gives the second item of the definition~\ref{d.singular-domination} for the maximal invariant set $\Lambda^0_Y$.
\medskip

We then compactifies the sets $\Lambda^0_Y\setminus \sing(Y)$ for $Y\in \cU_0$.
For each $\sigma\in\sing(Y)\cap V^-$, we denote by $K^-(\sigma)$ the projective space of the linear space $F(\sigma)$.
Analogously for $\sigma\in\sing(Y)\cap V^+$ we denote $K^+(\sigma)$ the projective space of the linear space $E(\sigma)$.
We then define:
$$K_Y=\bigcup_{x\in\Lambda^0_Y\setminus\sing(Y)}\RR Y(x) \cup \bigcup_{\sigma\in\sing(Y)\cap\Lambda^0_Y\cap V^-} K^-(\sigma)\cup \bigcup_{\sigma\in\sing(X)\cap\Lambda^0_Y\cap V^+} K^+(\sigma).$$

\begin{Claim}
$K_Y$ is a compact $(\wh\varphi^Y_t)_{t\in\RR}$-invariant set which varies upper semi-continuously with respect to the vector fields $Y\in \cU_0$.
\end{Claim} 
\proof
By definition, $K_Y$ is $(\wh\varphi^Y_t)_{t\in\RR}$-invariant. In order to prove the compactness and the semi-continuity,
it suffices to prove that, for any sequence $Y_n\to Y$ and for any  sequence of points $x_n\subset\Lambda_{Y_n}^0\setminus\sing(Y_n)$ converging to $\sigma\in \sing(Y)\cap V^-$ (resp. $\sigma\in \sing(Y)\cap V^+$), each limit of $\RR Y_n(x_n)$ belongs to $K^-(\sigma)$ (resp. $K^+(\sigma)$).  We only consider the case where  $x_n$ tends to $\sigma\in \sing(Y)\cap V^-$ since the other case is analogous.

Assume by contradiction, that $\RR Y_n(x_n)$ converges to a line $L$ which is not contained in $F(\sigma)$.
Using the domination $T_\sigma M=E^{ss}\oplus F$ over $\sigma$ (and up to considering a subsequence),
there exists $y_n$ in the backward orbit of $x_n$ such that $\RR Y_n(y_n)$ converges to a line $L'\subset E^{ss}(\sigma)$.
On a small open neighborhood $V$ of $\sigma$, there exists a continuous $(D\varphi_t)_{t<0}$-invariant cone field $\cC^{ss}$
such that any vector in $\cC^{ss}$ is uniformly expanded by $(D\varphi_t)_{t<0}$. For $n$ large,  $Y_n(y_n)$ is tangent to $\cC^{ss}$:
this implies that the backward orbit of $y_n$ escapes from $V$.
Let $t_n>0$ be the smallest number such that $\varphi_{-t_n}(y_n)$ is not in $V$ and let $z$ be an accumulation point of $\varphi_{-t_n}(y_n)$.
The backward invariance of the cone field shows that $z$ belongs to $W^{ss}(\sigma)\setminus \{\sigma\}$.
It also belongs to $\Lambda$, and this contradicts to the fact that $W^{ss}(\sigma)\cap\Lambda=\{\sigma\}.$
\endproof

Since the existence of a dominated splitting for continuous linear cocycles over compact spaces is robust,
there exists a $(\eta,T)$-dominated splitting of index $i$ over $K_Y$ for the extended linear Poincar\'e flow for any $Y$ in a $C^1$-neighborhood $\cU$ of $X$.
In particular the linear Poincar\'e flow over $\Lambda_Y$ admits a $(\eta,T)$-dominated splitting for any $Y\in \cU$.
\endproof

\subsection{Robust dominated splitting implies singular dominated splitting}
We now prove a converse statement to Theorem~\ref{p.robust-singular-domination}. Note that the second assumption is very mild
(it is satisfied once $X$ is Kupka-Smale and $\Lambda$ is chain transitive).

\begin{Proposition}~\label{p.robust-domination-imply-singular domination}
Let $X\in \cX^1(M)$ and $\Lambda$ be a  compact invariant set such that:
\begin{itemize}
\item $\Lambda$ admits a robust dominated splitting of index $i$:
there exist  $\eta,T>0$ and neighborhoods $\cU$ of $X$ and $U$ of $\Lambda$,
such that, for any $Y\in \cU$, the maximal invariant set in $U$ admits a $(\eta,T)$-dominated splitting of index $i$ for the linear Poincar\'e flow.
\item Each $\sigma\in \Lambda\cap \sing(X)$ is hyperbolic; moreover $W^{s}(\sigma)\cap\Lambda\setminus\{\sigma\}\neq\emptyset$ and  $W^{u}(\sigma)\cap\Lambda\setminus\{\sigma\}\neq\emptyset$.
\end{itemize}
Then the set $\Lambda$ admits a singular dominated splitting of index $i$.
\end{Proposition}

\proof
Let $\widetilde{\Lambda}$ be the closure of $\{\RR X(x):x\in \Lambda\setminus\sing(X) \}$ in $\cG^1$.
The extended linear Poincar\'e flow is a continuous cocycle over the compact space $\cG^1$.
As a consequence the dominated splitting over $\Lambda\setminus \Sing(X)$ for the linear Poincar\'e flow extends over $\widetilde\Lambda$ for the extended linear Poincar\'e flow. 
Consider a singularity $\sigma\in \Lambda$. Without loss of generality, one can assume $\dim(E^{s}(\sigma))\geq i$.
By assumption, there exists a line $L\in\widetilde\Lambda$ which is contained in $E^u(\sigma)$. Then, from the first item of Proposition~\ref{p.strong-stable-outside},
there exists a dominated splitting $E^s(\sigma)=E^{ss}(\sigma)\oplus E^{cs}(\sigma)$ with $\dim(E^{ss}(\sigma))=i.$

It remains to show that $W^{ss}(\sigma)\cap\Lambda=\{\sigma\}.$
This is proved by contradiction: we assume that there exists a point $y\in W^{ss}(\sigma)\cap\Lambda\setminus \{\sigma\}$.
\begin{Claim}
There exist a sequence $X_n\to X$ in $\cX^1(M)$ and a sequence $x_n\to \sigma$ in $M$ such that:
\begin{itemize}
\item $\RR X_n(x_n)$ converges to some line in $E^{ss}(\sigma)\oplus E^u(\sigma)\setminus\{E^{ss}(\sigma)\cup E^u(\sigma) \}$;
\item
$x_n$ belongs to the maximal invariant set of ${X_n}$ in $U$.
\end{itemize}
\end{Claim}
\proof  We consider three cases.
\smallskip

\noindent
{\it Case 1. The $\alpha$-limit set of $y$ intersects $W^{u}(\sigma)\setminus \{\sigma\}$ at a point $z$.}
The connecting lemma (Theorem~\ref{thm.connecting-lemma}) gives $C^1$-perturbations $X'$ of $X$
which coincide with $X$ on $\{\varphi_t(y), t>0\}\cup \{\varphi_{-t}(z), t>0\}\cup\{\sigma\}$
and such that the orbits of $y$ and $z$ for $X'$ coincide and are contained in $U$.

One can consider a small chart $\psi_\sigma\colon T_\sigma M\to V$ on a neighborhood of $\sigma$
such that $\psi(0)=\sigma$ and such that $\psi^{-1}(W_{loc}^{ss}(\sigma))$ and $\psi^{-1}(W^u_{loc}(\sigma))$
coincide with the linear spaces $E^{ss}$ and $E^u$.
By an arbitrarily small $C^1$-perturbation (in a neighborhood of $\sigma$), one can furthermore assume that $X'$
is linear on a neighborhood of $0$.
Another small $C^1$-perturbation near $y$ and $z$  gives a vector field $X''$ such that 
	\begin{itemize}
		\item $X''=X'$ on a small neighborhood of $\sigma$; 
		\item  $y,z$ are on a same periodic orbit in $U$ under $X''$
		which contains a piece of orbit of the linear vector field which is included in $E^{ss}(\sigma)\oplus E^u(\sigma)$ with points close to $\sigma$.
	\end{itemize}
One deduces that there exists a point $x$ on the periodic orbit of $y$ and $z$ under $X''$ which is close to $\sigma$
such that $\frac{X''(x)}{\|X''(x)\|}\in E^{ss}(\sigma)\oplus E^u(\sigma)\setminus\{E^{ss}(\sigma), E^u(\sigma) \}.$
\medskip

\noindent
{\it Case 2. There exists a point $z\in W^u(\sigma)\cap \Lambda\setminus \{\sigma\}$ whose $\omega$-limit set intersects
	$W^{ss}(\sigma)\setminus \{\sigma\}$.} This case is analogous to the case 1.
\medskip

\noindent
{\it Case 3. $\alpha(y)\cap W^u(\sigma)=\emptyset$ and there exists $z\in W^u(\sigma)\cap \Lambda\setminus \{\sigma\}$
		such that $\omega(z)\cap W^{s}(\sigma)=\emptyset$.}
		As in the first case, one considers a chart $\psi_\sigma$ and a $C^1$-close vector field $X'$ which is linear near~$0$.
		Another small $C^1$-perturbation near $y$ and $z$  gives a vector field $X''$
		such that 
		\begin{itemize}
			\item  $X''=X'$ on a neighborhood of $\sigma$, on the positive orbit of $z$ and negative orbit of $y$;
			\item  $z$ belongs to the positive orbit of $y$ and the orbit segment from $y$ to $z$; 
			\item  there is a point $x$ in the orbit segment from $y$ to $z$ contains a piece of orbit of the linear vector field which is included in $E^{ss}(\sigma)\oplus E^u(\sigma)$ 			with points close to $\sigma$.
		\end{itemize}
One concludes as in the first case.

\medskip

\noindent
{\it Case 4. $\alpha(y)\cap W^u(\sigma)=\emptyset$ and there is $z\in W^u(\sigma)\cap \Lambda\setminus \{\sigma\}$
		such that $\omega(z)\cap W^{s}(\sigma)\setminus W^{ss}(\sigma)\neq\emptyset$.}  Then by connecting lemma, one gets a $C^1$-perturbation $\widetilde X$ of $X$ such that 
		\begin{itemize}
			\item $\widetilde X$ coincides with  $X$ on $\{\varphi_t(y)\}_{t\in\RR}$ and $y$ belongs to $W^{ss}(\sigma)$;
			\item  There exists a point $z'\in W^u(\sigma)\cap W^s(\sigma)$.
		\end{itemize}
One then concludes as in Case 3. 	
\endproof

Let $\overline{\Lambda}_n$ be the closure of $\{\RR X_n(\varphi^{X_n}_t(x_n)),\; t\in\RR \}$ in $\cG^1$. By assumption, for $n$ large, the set $\overline{\Lambda}_n$ admits a $(\eta, T)$-domination of index $i$ for the extended  linear Poincar\'e flow $(\Psi_t^{X_n})$.  By Appendix B.1 in \cite{BDV},  this domination of index $i$ can be passed to the set $\wh\Lambda$ for the extended linear Poincar\'e flow $(\Psi_t^X)$, where $\wh\Lambda$ is the limit supremum of $\overline{\Lambda}_n$.  By the Claim above,  there exists $L\subset E^{ss}(\sigma)\oplus E^u(\sigma)$, not contained in $E^{ss}(\sigma)\cup E^u(\sigma)$  which belongs to $\wh\Lambda$
and this contradicts   Proposition~\ref{p.strong-stable-outside}.
\endproof
\begin{Remark}
There exists a vector field $X$ admitting a chain transitive invariant compact set which is dominated for linear Poincar\'e flow but which is not singular dominated.

\rm By the previous proposition the domination is not robust. This example is built by considering a vector field $X$ on a 3-manifold with a hyperbolic singularity $\sigma$ of index $2$ such that:
\begin{itemize}
	\item $\sigma$ admits a dominated splitting $T_\sigma M=E^{ss}\oplus E^{cs}\oplus E^u$: there exist $\lambda_{ss}<\lambda_s<0<\lambda_u$ and an isometric chart $\varphi\colon (-1,1)^3\to U$ on a neighborhood of $\sigma$ where $X$ has the form
	$$X(x_{ss},x_{s},x_{u})=(\lambda_{ss} x_{ss},\lambda_s x_s,\lambda_{u}x_u).$$
	\item  $W^{ss}(\sigma)$ and $W^u(\sigma)$ have a non-empty intersection along a regular orbit $\{\varphi_t(y)\}_{t\in\RR}$.
	\item There exist a point $z$ (respectively $w$) in the orbit of $y$ whose backward (respectively forward) orbit stays in $U$
	and local sections $S_z$ (respectively $S_w$) at $z$ (respectively $w$) by discs orthogonal to $X$ such that the holonomy map of the flow
	has the form
	$$(x_{ss},x_{s})\mapsto(x'_s,x'_u)=(x_{s},x_{ss}).$$
\end{itemize}
\end{Remark}

\section{Multi-singular hyperbolicity}~\label{s.multi-singular-hyperbolic}
In this section we prove that the multi-singular hyperbolicity is robust (Theorem~\ref{thm.robustness-of-multi-singular-hyperbolicity}),
we discuss the Lorenz-like property of the singularities and we compare with the uniform hyperbolicity and the singular hyperbolicity
(Theorem~\ref{t.comparison}).
 
 \subsection{Preparation}
We first state a basic result which will be used   in this paper.

\begin{Lemma}~\label{l.sub-additive-function} Let us consider a continuous flow $(\varphi_t)_{t\in\RR}$ on a compact metric space $K$,
	a one-parameter family $\{a_t\}_{t\in\RR}$ of continuous functions $K\to \RR$ and numbers $\{c_t\}_{t\geq 0}$ satisfying
	$$a_{t+s}(x)\leq a_s(x)+a_t(\varphi_s(x))\quad  \textrm{ for any $x\in K$ and any $t,s\in\RR$},$$
	$$ \sup_{s\in[-t,t]} a_s(x)<c_t\quad  \text{ for any $x\in K$ and any $t\geq 0$}.$$  
	Then for any $T>0$ and any orbit segment $\{\varphi_s(x) \}_{s\in[0,t]}$ with $t\geq 3T$, one has  
	$$a_t(x)\leq 3 c_T+\frac{1}{T}\int_{0}^ta_T(\varphi_s(x))\ud s .$$
\end{Lemma}
\begin{proof} By our assumptions $a_0(x)=0$ for any $x\in K$. Hence $c_t>0$ for any $t\geq 0$.
	Given $T>0$,  for any $t\geq 3 T$ and any  $s\in[0,T]$, one has 
		$$a_t(x)\leq a_s(x)+\sum_{i=0}^{[\frac{t}{T}]}a_T(\varphi_{s+iT}(x))+a_{t-s-([\frac{t}{T}]+1)T}(\varphi_{s+([\frac{t}{T}]+1)T}(x))
		\leq 2 {c_T}+\sum_{i=0}^{[\frac{t}{T}]} a_T(\varphi_{s+iT}(x)).
		$$
		Then one integrates over the interval $[0,T]$ and divides it by $T$:
		\begin{align*}
		a_t(x) \;&\leq\; 2\cdot {c_T}+\frac{1}{T}\int_{0}^T\sum_{i=0}^{[\frac{t}{T}]} a_T(\varphi_{s+iT}(x))\ud s
		\;\leq\; 2\cdot {c_T}+\frac{1}{T}\int_{0}^{([\frac{t}{T}]+1)T} a_T(\varphi_{s}(x))\ud s
		\\
		\;&\leq\; 3\cdot {c_T}+\frac{1}{T}\int_{0}^{t} a_T(\varphi_{s}(x))\ud s.
		\end{align*}
\end{proof}

\begin{Proposition}~\label{p.hyperbolicity-of-erogodic-measure-in-multi-sing}
 	Let $X\in\cX^1(M)$ and $\Lambda$ be a multi-singular hyperbolic set. Then each regular measure $\mu$ supported on $\Lambda$ is hyperbolic, and its hyperbolic splitting for $(\Psi_t)_{t\in\RR}$ coincides with the singular domination $\cN|_{\Lambda\setminus\sing(X)}=\cN^s\oplus\cN^u$.
	Moreover, there exists $\eta>0$ such that for any regular invariant measure $\mu$ supported on $\Lambda$ and for any $T>0$ large enough,
 	$$\frac{1}{T}\int\log\|\Psi_{T}|_{\cN^s}\|\ud\mu<-\eta \quad \textrm{~and~} \quad
 \frac{1}{T}\int\log \|\Psi_{-T}|_{\cN^u}\|\ud\mu<-\eta.$$ 
 \end{Proposition}
\proof
Let $\cN|_{\Lambda\setminus\sing(X)}=\cN^s\oplus\cN^u$ be the singular domination over $\Lambda$ for $(\Psi_t)_{t\in\RR}$, let $V$ be the closed neighborhood of $\sing(X)\cap \Lambda$ and let $\eta_0, T_0>0$ be the numbers as in Definition~\ref{Def:multi-singular-hyperbolic}.

Given a regular ergodic measure $\mu$ supported on $\Lambda$, since the maximal invariant set in $V$ is $\Lambda\cap \sing(X)$,
there exists an open set $U$ which is disjoint from $V$ and satisfies $\mu(U)>0$. 
Now, by Oseledec theorem and Poincar\'e recurrence theorem,  one can choose $x\in U\cap\Lambda$ such that
\begin{itemize}
	\item $\lim_{t\rightarrow+\infty}\frac{1}{t}\log\|\Psi_t|_{\cN^s(x)}\|$ is the maximal Lyapunov exponent of $\mu$ along $\cN^s$;
	\item there exists $t>0$ arbitrarily large such that $\varphi_t(x)\in U$.
\end{itemize} 
The second item above and Definition~\ref{Def:multi-singular-hyperbolic} give that there exists $t>T_0$ arbitrarily large such that $\|\Psi_t|_{\cN^s(x)}\|<e^{-\eta_0 t}$, thus  the maximal Lyapunov exponent of $\mu$ along $\cN^s$ is no larger than $-\eta_0.$ Analogously, one can show that  the minimal Lyapunov exponent of $\mu$ along $\cN^u$ is no less than $\eta_0$, hence $\mu$ is hyperbolic.
The moreover part comes from the dominated convergence theorem and sub-additive ergodic theorem. 
\endproof

\subsection{Robustness of the multi-singular hyperbolicity: proof of Theorem~\ref{thm.robustness-of-multi-singular-hyperbolicity}}
By Theorem~\ref{p.robust-singular-domination}, there exist  a $C^1$-neighborhood $\cU_0$ of $X$ and a closed neighborhood $U_0$ of $\Lambda$ such that the maximal invariant set $\Lambda^0_Y$ in $U_0$ for any $Y\in \cU_0$ admits a singular domination. This gives the first item in Definition~\ref{Def:multi-singular-hyperbolic}.

The singularities of $X$ in $\Lambda$ are hyperbolic and will be denoted by $\sigma_1,\cdots,\sigma_\ell$.
Up to reducing $\cU_0$, one can assume that each singularity of $Y\in\cU_0$ in $U_0$ is the continuation of some $\sigma_i$.
In particular the third item in Definition~\ref{Def:multi-singular-hyperbolic} holds for $Y\in \cU_0$ and $\Lambda^0_Y$.
Up to changing the metric, one can also assume that the invariant spaces corresponding to  splitting $E^{ss}\oplus E^{c}\oplus E^{uu}$ over each singularity $\sigma_i$
for $X$ is orthogonal to each other.

Let $V$ be the neighborhood of $\{\sigma_1,\cdots,\sigma_\ell\}$ and $\eta, T>0$ be the numbers given in Definition~\ref{Def:multi-singular-hyperbolic}.
Since $V$ is compact, $V$ remains an isolating neighborhood of the continuation of singularities $\{\sigma_1,\cdots,\sigma_\ell\}$ for the $C^1$-close vector fields.
One only needs to check that there exist $T_0>T$, $\eta_0\in(0,\eta)$ and a small enough open neighborhood $U$ of $\Lambda$ such that for each vector field $Y$ which is  $C^1$ close to $X$, the second property of the definition holds  for the points in the maximal invariant set of $Y$ in $U$ with respect to the neighborhood $V$ and the numbers $\eta_0, T_0.$
In the following we consider the bundle $\cN^s$. The bundle $\cN^u$ can be handled in a similar way.
\medskip

The proof is proceeded by contradiction. We assume that there exist:
\begin{itemize}
	\item a sequence $(X_n)$ which converges to $X$ in $\cX^1(M)$;
	\item a sequence of positive numbers $t_n\to +\infty$;
	\item a sequence of points $(x_n)$  
\end{itemize}
which satisfy:
\begin{itemize}
	\item  the closure of $\orb(x_,\varphi_t^{X_n}(x_n))$ is contained in the $1/n$-neighborhood of $\Lambda$;
	\item $x_n,\varphi_{t_n}^{X_n}(x_n)\notin V \quad \text{and}\quad \|\Psi_{t_n}^{X_n}|_{\cN^s(x_n)}\|\geq e^{-\frac 1 n t_n}.$
\end{itemize}

Let us denote $L_n=\RR X_n(x_n)$ and let $\widehat \Lambda$ be the limit set of the orbits $\{\widehat \varphi_s^{X_n}(L_n)\}_{s\in \RR}$ in $\cG^1$ as $n\to +\infty$:
it is an invariant compact set which projects on an invariant subset of $\Lambda$.
Up to taking subsequences, there exists a probability invariant measure  $\widehat \mu$ on $\widehat \Lambda$
which projects on a measure $\mu$ on $M$, such that
$$\lim_{n\rightarrow\infty}\frac{1}{t_n}\int_0^{t_n}\delta_{\varphi_s^{X_n}(x_n)}\ud s=\mu \quad \text{and}\quad 
\lim_{n\rightarrow\infty}\frac{1}{t_n}\int_0^{t_n}\delta_{\widehat\varphi_s^{X_n}(L_n)}\ud s=\widehat\mu.$$

\begin{Claim-numbered}\label{c.slow-contraction}
	$ \frac{1}{\tau}\int\log\|\wh\Psi_\tau|_{\cN^s}\|\ud\wh\mu\geq 0$ for any $\tau>0$.
\end{Claim-numbered}
\proof
By the definition of extended linear Poincar\'e flow, $\Psi_{s}^{X_n}|_{\cN^s(\varphi_t(x_n))}=\wh\Psi_{s}^{X_n}|_{\cN^s(\wh\varphi_t(L_n))}$ for any $s,t\in\RR$.
Since $X_n$ converges to $X$ in $C^1$-topology and the norm of the time $t$-map of a linear Poincar\'e flow is bounded by the norm of the  time $t$-map of the tangent flow, for each $t>0$ there exists $c_t>0$ such that 
$$\sup_{n\in\NN}\sup_{L\in\cG^1}\sup_{s\in[-t,t]}\log\|\wh\Psi^{X_n}_s(L)\|<c_t.$$
Applying     Lemma~\ref{l.sub-additive-function} to the family of continuous functions $\big\{\log\|\wh\Psi_t^{X_n}|_{\cN^s}\|\big\}_{t>0}$, one gets 
$$\log\|\wh\Psi_{t_n}^{X_n}|_{\cN^s(L_n)}\|\leq 3 c_\tau +\frac{1}{\tau}\int_0^{t_n}\log\|\wh\Psi_\tau^{X_n}|_{\cN^s(\wh\varphi_s^{X_n}(L_n))}\|\ud s\textrm{~~for any $\tau>0$}.$$
The functions $\frac 1 t \log\|\wh\Psi_t^{X_n}|_{\cN^s}\|$ converge to $\frac 1 t \log\|\wh\Psi_t^{X}|_{\cN^s}\|$ as $n\to+\infty$, uniformly in $t$.
The choice of  orbit segment $\{\wh\varphi_s^{X_n}(L_n) \}_{s\in[0,t_n]}$ gives:
$$\limsup_{n\rightarrow\infty}\frac{1}{t_n}\log\|\wh\Psi_{t_n}^{X_n}|_{\cN^s(L_n)}\|\leq\frac{1}{\tau}\int\log\|\wh\Psi_\tau|_{\cN^s}\|\ud\wh\mu.
$$
As $\|\wh\Psi_{t_n}^{X_n}|_{\cN^s(L_n)}\|=\|\Psi_{t_n}^{X_n}|_{\cN^s(x_n)}\|\geq e^{-\frac 1 n t_n}$, one gets the announced inequality.
\endproof

One can decompose $\widehat\mu$ as the barycenter of three invariant probability measures:
$$\widehat\mu=\alpha\cdot\widehat\nu+\beta\cdot \widehat\nu^++\gamma\cdot\widehat\nu^-,$$
where $\widehat \nu,\widehat \nu^+,\widehat \nu^-$ projects to measures $\nu,\nu^+,\nu^-$ on $M$ such that $\nu$ is regular
and $\nu^+$ (resp. $\nu^-$) is supported on the set of singularities $\sigma$ such that $W^{ss}(\sigma)\cap \Lambda\setminus \{\sigma\}=\emptyset$
(resp. $W^{uu}(\sigma)\cap \Lambda\setminus \{\sigma\}=\emptyset$).
We study independently each of these measures.

\paragraph{\it The measure $\widehat \nu$.}
The Proposition~\ref{p.hyperbolicity-of-erogodic-measure-in-multi-sing} implies
$ \frac{1}{\tau}\int\log\|\wh\Psi_\tau|_{\cN^s}\|\ud\wh\nu< 0$ for any $\tau>0$ large.

\paragraph{\it The measure $\widehat \nu^+$.}
Let us consider a singularity $\sigma\in \Lambda$ in the support of $\nu^+$.
Since $W^{ss}(\sigma)\cap \Lambda\setminus \{\sigma\}=\emptyset$,
its preimage in $\widehat \Lambda$ is contained in the projectivization $K^+(\sigma)$ of the space $E^c(\sigma)\oplus E^{uu}(\sigma)$.
Since $E^{ss}(\sigma)$ is orthogonal to $E^c(\sigma)\oplus E^{uu}(\sigma)$,
the bundle $\cN^s$ above $K^+(\sigma)$ coincides with the space $E^{ss}(\sigma)$. This implies:
$$ \frac{1}{\tau}\int\log\|\wh\Psi_\tau|_{\cN^s}\|\ud\wh\nu^+< 0\quad \text{for any $ \tau>0$ large.}$$

\paragraph{\it The measure $\widehat \nu^-$.}
We then consider a singularity $\sigma\in \Lambda$ in the support of $\nu^-$.
Since $W^{uu}(\sigma)\cap \Lambda\setminus \{\sigma\}=\emptyset$,
and since $\sigma$ is accumulated by points $\varphi^{X_n}_{s_n}(x_n)$ with $0<s_n<t_n$,
the center direction $E^c(\sigma)$ has to be unstable.
Note also that the preimage of $\sigma$ in $\widehat \Lambda$ is contained in the projectivization $K^-(\sigma)$ of the space $E^{ss}(\sigma)\oplus E^{c}(\sigma)$
and that above $K^-(\sigma)$ the bundle $\cN^s$ is contained in $E^{ss}(\sigma)\oplus E^{c}(\sigma)$.
As $\sigma$ is Lorenz-like, there exist $\eta_\sigma>0$ and $T_\sigma>0$ such that  $\|D\varphi_t|_L\|\cdot \|\wh\Psi_t|_{\cN^s(L)}\|<e^{-2\eta_\sigma t}$ for any $t>T_\sigma$
and $L\in K^-(\sigma)$.
	
	Let us fix $\varepsilon>0$ and $\tau>T_\sigma$. 
	There exists an open neighborhood $V_\sigma$ of $K^-(\sigma)$ such that
	\begin{itemize}
		\item[(a)]
		$\widehat \mu(V_\sigma\setminus K^-(\sigma))\cdot \max (|\log\|\wh \Psi_\tau\||)< \varepsilon$,
		\item[(b)]  for any $L\in V_\sigma$, one has $\|D\varphi_\tau|_L\|\cdot \|\wh\Psi_\tau|_{\cN^s(L)}\|<e^{-\frac 32\eta_\sigma\tau}.$
	\end{itemize}
	This in particular implies that for $n$ large and any $s\in[0,t_n]$ with $\wh\varphi^{X_n}_s(L_n)\in V_\sigma$,  one has 
	\begin{eqnarray}~\label{eqn:sectional-contraction}
	\|D\varphi^{X_n}_\tau|_{\wh\varphi^{X_n}_s(L_n)}\|\cdot \|\wh\Psi_\tau|_{\cN^s(\wh\varphi^{X_n}_s(L_n))}\|<e^{-\eta_\sigma\tau}.
	\end{eqnarray}

Let us fix $\delta>0$ small. For  each $n$, 
	one can  introduce finitely many intervals $I^1_n,\dots,I_n^{m_n}$ that are the connected components of the set $\{s\in [0,t_n], \widehat \varphi_s^{X_n}(L_n)\in V_\sigma\}$
	such that $\wh\varphi_{I^i_n}(L_n)$ meets the $\delta$-neighborhood of $\sigma$. From (a), for $n$ large enough one has
	\begin{eqnarray}~\label{eqn:approximation-of-Ns}
	\bigg|\frac{1}{t_n}\sum_{i=1}^{m_n}\int_{s\in I_n^i} \log\|\wh\Psi^{X_n}_\tau|_{\cN^s(\varphi_s^{X_n}(L_n))}\|\ud s-\int \log\|\wh\Psi_\tau|_{\cN^s}\|\ud (\gamma\cdot\wh\nu^-|_{K^-(\sigma)})\bigg|<2\e.
	\end{eqnarray}	

\begin{Claim-numbered}
If $\delta$ is small enough, then for $n$ large
	$
	\frac{1}{t_n }\sum_{i=1}^{m_n}\int_{s\in I_n^i}\log\|D\varphi_\tau^{X_n}|_{\RR X_n(\varphi_s(x_n))}\|\ud s
	>-\varepsilon.$
	\end{Claim-numbered}
\begin{proof}
Let us denote $I_n^i=(a_n^i,b_n^i)$. Then one has 
	\begin{align*}
	\frac{1}{t_n }\sum_{i=1}^{m_n}\int_{s\in I_n^i}\log\|D\varphi_\tau^{X_n}&|_{\RR X_n(\varphi_s(x_n))}\|\ud s
	=\frac{1}{t_n}\sum_{i=1}^{m_n}\int_{a_n^i}^{b_n^i}\log\|X_n(\varphi_{s+\tau}(x_n))\|-\log\|X_n(\varphi_s^{X_n}(x_n))\|\ud s\\
	&\leq \frac{1}{t_n}\sum_{i=1}^{m_n}\bigg(\int_{b_n^i}^{b_n^i+\tau}\log\|X_n(\varphi_{s}(x_n))\|\ud s-\int_{a_n^i}^{a_n^i+\tau}\log\|X_n(\varphi_s^{X_n}(x_n))\|\ud s\bigg).
	\end{align*}
	Since $\varphi_{a_n^i}^{X_n}(x_n),\varphi_{b_n^i}^{X_n}(x_n)\in \partial V_\sigma$, there exists $c>0$ such that $$\sup_{n\in\NN}\sup_{x\in\partial V_\sigma,s\in(-\tau,\tau)}\log\|X_n(\varphi_s(x))\|<c.$$
If $\delta>0$ is small, then $|I_n^i|$ is arbitrarily large, so that
$$\frac{1}{|I_n^i|}\bigg(\int_{b_n^i}^{b_n^i+\tau}\log\|X_n(\varphi_{s}(x_n))\|\ud s-\int_{a_n^i}^{a_n^i+\tau}\log\|X_n(\varphi_s^{X_n}(x_n))\|\ud s\bigg)$$
is arbitrarily close to $0$.
\end{proof}

	By Equations~\eqref{eqn:sectional-contraction},~\eqref{eqn:approximation-of-Ns} and the previous claim, for $n$ large, one has 
	\begin{align*}
	\int\log\|\wh\Psi_\tau|_{\cN^s}\|\ud (\gamma\cdot\wh\nu^-|_{K^-(\sigma)})
	&\leq \frac{1}{t_n }\sum_{i=1}^{m_n}\int_{s\in I_i} \log\|\wh\Psi^{X_n}_\tau|_{\cN^s(\varphi_s^{X_n}(x_n))}\|+\log\|D\varphi_\tau^{X_n}|_{\RR X_n(\varphi_s(x_n))}\|\ud s +3\e
	\\
	&\leq -\eta_\sigma\cdot \tau \frac{1}{t_n}\sum_{i=1}^{m_n}|I_n^i|+3\e \leq -\eta_\sigma \big(\gamma-\e\big)+3\e<0.
	\end{align*}
	This proves that $\int\log\|\wh\Psi_\tau|_{\cN^s}\|\ud (\wh\nu^-|_{K^-(\sigma)})<0$ for $\tau>0$ large enough. 
	As there are only finitely many singularities in $\Lambda$, this gives
	$\int\log\|\wh\Psi_\tau\|\ud\wh\nu^-<0$ for $\tau>0$ large.
\medskip

To summarize, there exists $\tau>0$ arbitrarily large such that 	$\frac 1 \tau\int\log\|\wh\Psi_\tau\|\ud\wh\mu<0$ which contradicts Claim~\eqref{c.slow-contraction}.
Theorem~\ref{thm.robustness-of-multi-singular-hyperbolicity} is now proved.
\endproof

 \subsection{A criterion for Lorenz-like singularities}
 In the definition of multi-singular hyperbolicity we require the singularities to be Lorenz-like.
 This is often a consequence of the other properties of the definition.
 
 \begin{Proposition}~\label{p.lorenz-like-for-non-isolated-sing}
 	Let $X\in\cX^1(M)$, let $\Lambda$ be an invariant compact set satisfying (i) and (ii) in Definition~\ref{Def:multi-singular-hyperbolic} and let $\sigma$ be a hyperbolic singularity in $\Lambda$.
 	If there exist  a sequence $(y_n)_{n\in\NN}$ in $\Lambda$ and a neighborhood $V_\sigma$ of $\sigma$ such that:
 	\begin{itemize}
 		\item $y_n$ tends to $\sigma$;
 		\item the forward and backward orbit of $y_n$ intersects $M\setminus V_\sigma.$
 	\end{itemize}  Then $\sigma$ is Lorenz-like.
 \end{Proposition}
 \proof
 Let $\cN|_{\Lambda\setminus\sing(X)}=\cN^s\oplus\cN^u$ be the singular domination over $\Lambda$ for $(\Psi_t)_{t\in\RR}$ and let $i$ be its index.  Let $\sigma\in\Lambda$ be a  singularity as in the assumption. Then $W^s(\sigma)\cap\Lambda\setminus\{\sigma\}\neq\emptyset$ and $W^u(\sigma)\cap\Lambda\setminus\{\sigma\}\neq\emptyset.$ 
 By the definition of singular domination, 
 one gets that 
 \begin{itemize}
 	\item if $\ind(\sigma)>i$, then  $E^{s}(\sigma)=E^{ss}(\sigma)\oplus E^{cs}(\sigma)$ with $\dim(E^{ss}(\sigma))=i$ and $W^{ss}(\sigma)\cap\Lambda=\{\sigma \}$; 
 	\item if $\ind(\sigma)\leq i$, then $E^{u}(\sigma)=E^{cu}(\sigma)\oplus E^{uu}(\sigma)$ with $\dim(E^{uu}(\sigma))=\dim(M)-1-i$ and $W^{uu}(\sigma)\cap\Lambda=\{\sigma \}$.
 \end{itemize}

 Without loss of generality, from now on, we assume  that $\sigma$ admits a partially hyperbolic splitting for $(D\varphi_t)_{t\in \RR}$ of the form $E^{ss}\oplus E^{cs}\oplus E^u$ where $\dim(E^{ss})=i$ and $W^{ss}(\sigma)\cap\Lambda=\{\sigma\}$.

 Up to changing the  Riemannian metric, one can assume that each bundle in the splitting $E^{ss}\oplus E^{cs}\oplus E^u$ is orthogonal to the other. Let $\lambda_1\leq \lambda_2\leq \cdots\leq\lambda_k$ be all the Lyapunov exponents of $\sigma$ along $E^{cs}\oplus E^u$, where $k=\dim(E^{cs}\oplus E^u).$
 It remains to prove that $E^{cs}\oplus E^u$ is sectionally expanding under $(D\varphi_t)_{t\in\RR},$ that is, $\lambda_1+\lambda_2>0$, which in return  implies that $\dim(E^{cs})=1$.  
 
 Let $\lambda=\lambda_1+\lambda_2$.
 For $\e>0$, consider a small neighborhood $U_\sigma\subset V_\sigma$ of $\sigma$ where one can define a cone field $\cC^{cu}$ with respect to $E^{cs}\oplus E^u$ such that 
 \begin{itemize}
 	\item $\cC^{cu}$ is $(D\varphi_t)_{t>T}$-invariant for some constant $T>0$;
 	\item there exists $c>1$ such that for any $x\in U_\sigma$, any $t>T$ satisfying $\{\varphi_s(x)\}_{s\in[0,t]}\subset U_\sigma$  and any $\dim(E^{cs}\oplus E^u)$-dimensional linear space $F\subset \cC^{cu}(x)$, there exists $2$-plane $P\subset F$ so that  
 	$$\det(D\varphi_t|_F)\leq c\cdot e^{(\lambda+\e)t}.$$
 \end{itemize} 
 Moreover since $W^{ss}(\sigma)\cap\Lambda=\{\sigma\}$, we can assume that $\RR X(x)\oplus\cN^u(x)\subset\cC^{cu}(x)$ for any $x\in\Lambda\cap U_\sigma$.
 
 Consider the sequence of points $y_n$ as in the assumption.  Let $x_n$ be the first intersection of the backward orbit of $y_n$ with the boundary of $U_\sigma$ and let $t_n>0$ the first time that the forward orbit of $x_n$ intersects the boundary of $U_\sigma$. In particular, 
 $$x_n,\varphi_{t_n}(x_n)\in\partial U_\sigma\textrm { and   } y_n\in \{\varphi_s(x_n) \}_{s\in(0,t_n)}\subset U_\sigma.$$
 Hence there exists a $2$-plane $P\subset \RR X(x_n)\oplus \cN^u(x_n)$ such that 
 \begin{eqnarray}~\label{eq:sectional-expanding-rate}
 \det(D\varphi_{t_n}|_{P})\leq c\cdot e^{(\lambda+\e)t_n}.
 \end{eqnarray}
 
 Let $V$ be the compact isolating neighborhood in (ii) from Definition~\ref{Def:multi-singular-hyperbolic}. Then there exists $l>0$ such that for any $x\in\partial U_\sigma$ and $t>0$ with $\varphi_t(x)\in\partial U_\sigma$ and $\{\varphi_s(x)\}_{s\in(0,t)}\subset U_\sigma$, the forward orbit of $\varphi_t(x)$ and the backward orbit of $x$ leave $U_\sigma$ in time smaller than $l.$ Therefore, there exist $s_n,\tau_n\in(-0,l)$ such that $\varphi_{-s_n}(x_n)\in\partial V$ and $\varphi_{t_n+\tau_n}(x_n)\in\partial V.$ 
 From (ii) in Definition~\ref{Def:multi-singular-hyperbolic}, there exist $c,\eta>0$ such that
  \begin{eqnarray}~\label{eq:sectional-expanding-rate-2}
 \det(D\varphi_{t_n+s_n+\tau_n}|_{P})\geq c\cdot e^{\eta (t_n+s_n+\tau_n)}\textrm{~for any $2$-plane  $P\subset \RR X(\varphi_{-s_n}(x_n))\oplus \cN^u(\varphi_{-s_n}(x_n))$}.
 \end{eqnarray}
 As $s_n$ and $\tau_n$ are uniformly bounded, by Equations~\eqref{eq:sectional-expanding-rate} and~\eqref{eq:sectional-expanding-rate-2}, one has $\lambda+\e\geq \eta$. The arbitrariness of $\e$ implies that $\lambda\geq\eta>0.$
 \endproof 
 \subsection{Uniform and singular hyperbolicities: proof of Theorem~\ref{t.comparison}}
 
 \paragraph{a.} We prove the first item of Theorem~\ref{t.comparison}. Let us consider a uniformly hyperbolic set $\Lambda$
 such that for each $\sigma\in \sing(\Lambda)$ both $W^{s}(\sigma)\cap \Lambda\setminus \{\sigma\}$
 and $W^{u}(\sigma)\cap \Lambda\setminus \{\sigma\}$ are non empty.
 We first prove that $\sing(X)=\emptyset$ since the uniform hyperbolicity
 along regular orbits in $W^s(\sigma)$ and $W^u(\sigma)$ gives incompatible stable dimension at $\sigma$.
 The restriction of the splittings $E^s\oplus \RR X$ and $\RR X\oplus E^u$ to the normal bundle induces a dominated splitting of the linear Poincar\'e flow
 which satisfies the definition of multi-singular hyperbolicity.
 
 \paragraph{b.} Conversely, if $\Lambda$ is a multi-singular hyperbolic set which does not contain any singularity, the item (ii) in Definition~\ref{Def:multi-singular-hyperbolic}
 shows that the bundles $\cN^s$ and $\cN^u$ are respectively uniformly contracted and uniformly expanded by the linear Poincar\'e flow,
 whereas the action of the differential on the bundle $\RR X$ remains bounded.
 Then Proposition~\ref{p.domination-criterion}  implies that the tangent bundle over $\Lambda$ admits dominated splittings
 $TM=E^s\oplus F=E\oplus E^u$ with $\dim(E^s)=\dim(\cN^s)$, $\dim(E^u)=\dim(\cN^u)$.
 The existence of a finest dominated splitting (see~\cite[Appendix B.1]{BDV}) then gives a dominated splitting
 $TM=E^s\oplus E^c\oplus E^u$ with $\dim(E^c)=1$. Since the invariant bundle $\RR X$ remains bounded,
 it remains in uniform cones transverse to $E^s$ and $E^u$. The invariance and the domination then give
 $E^c=\RR X$, proving that $\Lambda$ is uniformly hyperbolic. The proof of the first item is complete.
 
\paragraph{c.}  We now turn to the second item of Theorem~\ref{t.comparison} and consider an invariant compact set $\Lambda$ which is singular hyperbolic.
 We will assume for instance that it has a dominated splitting of the form $TM|_\Lambda=E^{ss}\oplus E^{cu}$, as in Definition~\ref{Def:singular-hyperbolicity}.

We first notice that at each point $x\in \Lambda$ we have $X(x)\in E^{cu}(x)$.
 Indeed if one assumes by contradiction that $x$ is regular and satisfies $X(x)\not\in E^{cu}(x)$, the backward orbit of $x$ remains uniformly transverse to the bundle $E^{cu}$
 and avoids a neighborhood of the singularities. The $\alpha$-limit set of $x$ is thus non-singular and (by domination), the restriction of the vector fields $X$ is tangent to $E^{ss}$.
 This is a contradiction since for any probability measure on $\alpha(x)$,
 the Lyapunov exponent in the direction of the flow is not negative.

Since $X\in E^{cu}$, the linear Poincar\'e flow also admits a dominated splitting $\cN=\cN_1\oplus \cN_2$ of index $\dim(E^{ss})$,
obtained by intersecting $E^{ss}(x)\oplus \RR X(x)$ and $E^{cu}(x)$ with $\cN(x)$ at each regular point $x$.
Moreover for each singularity $\sigma\in \Lambda$, we also have $W^{ss}(\sigma)\cap \Lambda=\{\sigma\}$.
Consequently $\Lambda$ has a singular domination of index $\dim(E^{ss})$.
Since $E^{ss}$ is uniformly contracted, the bundle $\cN_1$ is uniformly contracted by the linear Poincar\'e flow.

Let $V$ be a neighborhood of $\sing(X)\cap \Lambda$.
For any regular point $x$ and any unit vector $v\in\cN_1$, the volume growth under the tangent flow $D\varphi_t$
along the plane spanned by $v$ and $X(x)$ is equal to $\|\Psi_t(x).v\|\frac{\|X(\varphi_t(x))\|}{\|X(x)\|}$:
the singular hyperbolicity implies that there exist $T_0,\eta>0$ such that for any $x$ and $t>T_0$,
this quantity is larger than $\exp(2\eta t)$. When $x,\varphi_t(x)$ are outside $V$, the quotient $\frac{\|X(\varphi_t(x))\|}{\|X(x)\|}$
is bounded away from $0$ by a constant $1/C$.
Choose $T>T_0$ such that $\exp(\eta T)>C$ which implies $\|\Psi_t(x).v\|\geq \exp(\eta t)$.
This concludes the second item of Definition~\ref{Def:multi-singular-hyperbolic}.

Since each singularity is hyperbolic and $W^s(\sigma)\cap \Lambda\setminus\{\sigma\}\neq \emptyset$,
there exists a stable direction inside $E^{cu}(\sigma)$. The singular hyperbolicity implies that $E^{cu}(\sigma)$
decomposes as $E^{cu}(\sigma)=E^c\oplus E^u$ with $\dim(E^c)=1$.
We have thus proved that each singularity is Lorenz like. This ends the proof that $\Lambda$
is multi-singular hyperbolic.

\paragraph{d.} Finally we consider a multi-singular hyperbolic set $\Lambda$ with a singular dominated splitting $\cN=\cN^s\oplus \cN^u$ as in Definition~\ref{Def:multi-singular-hyperbolic} and we assume that
at any singularity $\sigma\in\Lambda$:
\begin{itemize}
\item both $W^{s}(\sigma)\cap \Lambda\setminus \{\sigma\}$ and $W^{u}(\sigma)\cap \Lambda\setminus \{\sigma\}$ are non empty,
\item there exists a  dominated splitting $T_\sigma M=E^{ss}\oplus E^c\oplus E^{uu}$, where $E^{ss},E^{uu}$ have the same dimensions as $\cN^s, \cN^u$ and where $E^c$ is a stable line.
\end{itemize}
The case where $E^c(\sigma)$ is an unstable line for all singularity $\sigma$ can be handled analogously.
Up to changing the metric, the splitting can be assumed orthogonal at each space $T_\sigma M$.

The singular domination implies that at each singularity
either $W^{ss}(\sigma)\cap \Lambda\setminus \{\sigma\}$ or $W^{u}(\sigma)\cap \Lambda\setminus \{\sigma\}$ is empty.
Since $E^c$ is contracting and $W^{u}(\sigma)\cap \Lambda\setminus \{\sigma\}\neq \emptyset$, one concludes that
\begin{equation}\label{e.ss-outside}
W^{ss}(\sigma)\cap \Lambda\setminus \{\sigma\}=\emptyset .
\end{equation}
We then prove the uniform contraction of the bundle $\cN^s$ under the flow $(\Psi_t)$.
Let us consider any regular $x\in \Lambda$ and any $t>0$.

We choose $\eta_0,T_0>0$ and a small open neighborhood $V$ of $\sing(X)\cap \Lambda$ as in the item (ii) of Definition~\ref{Def:multi-singular-hyperbolic}:
if $x$ and $\varphi_t(x)$ do not belong to $V$ and $t\geq T_0$, one has $\|\Psi_t|_{\cN^s}(x)\|\leq \exp(-\eta_0 t)$.

If the orbit $(\varphi_s(x))_{s\in[0,t]}$ is contained in $V$, the property~\eqref{e.ss-outside} implies that
$\RR X(\varphi_s(x))$ is close to a line in $E^c\oplus E^{uu}$.
Then the dominated splitting $\cN^s\oplus \cN^u$ and the fact that $E^{ss}$ is orthogonal to $E^c\oplus E^u$
imply that $\cN^s(\varphi_s(x))$ is close to $E^{ss}(\sigma)$.
Consequently, there exist  $\eta_1,T_1>0$ such that if $t\geq T_1$, then $\|\Psi_t|_{\cN^s}(x)\|\leq \exp(-\eta_1 t)$.
We choose $C>0$ such that for any piece of orbit of length $t\leq \max(T_0,T_1)$,
we have $\|\Psi_t|_{\cN^s}(x)\|\leq C$. We also set $\eta=\min(\eta_0,\eta_1)$.

If the orbit segment $(\varphi_s(x))_{s\in[0,t]}$ is not entirely contained in $V$, we consider the largest interval $[t_1,t_2]\subset [0,t]$ such that
$\varphi_{t_i}(x)\notin V$ (we take $t_1=0$ provided that $x\in V$). Then the previous estimates give
$$\|\Psi_t|_{\cN^s}(x)\|\leq C^3\cdot \exp(3\eta(T_1+T_0))\cdot \exp(-\eta t).$$
This shows that $\cN^s$ is uniformly contracted by the linear Poincar\'e flow.
By Proposition~\ref{p.domination-criterion},
there exists a dominated splitting $TM|_\Lambda=E^{ss}\oplus F$ with $\dim(E^{ss})=\dim(\cN^s)$.

Any ergodic measure $\mu$ on  $\Lambda$ is
\begin{itemize}
\item either supported on a Lorenz-like singularity $\sigma$: by definition the sum of the two smallest Lyapunov exponents along $F(\sigma)=E^{cu}$ is positive,
\item or a regular measure having one vanishing Lyapunov exponent along $X$ and other positive Lyapunov exponents along $F$ due to Proposition~\ref{p.hyperbolicity-of-erogodic-measure-in-multi-sing}.
\end{itemize}
This implies that for the tangent flow above $\Lambda$, the volume along $2$-planes contained in $F$ grow exponentially and the set $\Lambda$ is singular hyperbolic.

The proof of Theorem~\ref{t.comparison} is now complete. \qed
 
 \section{Renormalization of multi-singular hyperbolicity}\label{s.equi-definition}
 In this section we compare Definition~\ref{Def:multi-singular-hyperbolic} with the definition given by Bonatti and  da Luz in~\cite{BdL}.
 We first recall some terminology about extended flows. 
 
 \subsection{The extended maximal invariant set}
 Let $X\in \cX^1(M)$, and $\Lambda$ be an invariant compact set. Let $\sigma\in \Lambda$ be a hyperbolic singularity and consider the finest dominated splitting for $(D\varphi_t)_{t\in\RR}$:
 $$T_\sigma M=E^s_k\oplus_{\prec} \cdots\oplus_\prec E^s_1\oplus_\prec E^u_1\oplus_\prec\cdots\oplus_\prec  E^u_l.$$
 Let $i$ be the smallest integer such that the strong stable manifold of $\sigma$ tangent to $E^s_k\oplus\cdots\oplus E^s_i$ intersects $\Lambda$
 only at $\sigma$. The space $E^{ss}_{\sigma,\Lambda}:=E^s_k\oplus\cdots\oplus E^s_i$ is called \emph{escaping stable space of $\sigma$ in $\Lambda$}.
 Analogously, we define the \emph{escaping unstable space of $\sigma$ in $\Lambda$}, and we denote it as $E^{uu}_{\sigma, \Lambda}:=E^s_j\oplus\cdots\oplus E^s_l$.  Now, the \emph{center space} at $\sigma$ is defined as 
 $E^c_{\sigma,\Lambda}=E^s_{i-1}\oplus \cdots \oplus E^s_1\oplus E^u_1\oplus \cdots\oplus E^u_{j-1}.$ We will denote by $\PP^c_{\sigma,\Lambda}$ the projective space of the center space $E^c_{\sigma,\Lambda}$. 
 
 \begin{Definition}
 	Let $X\in\cX^1(M)$ and $\Lambda$ be an invariant compact set whose singularities are all hyperbolic.
 	The \emph{extended invariant set of $\Lambda$} is the compact subset of $\cG^1$ defined by
 	$$B(X,\Lambda)=\overline{\big\{\RR X(x): x\in \Lambda\setminus\sing(X) \big\}}\cup\bigcup_{\sigma\in\sing(X)\cap \Lambda} \PP^c_{\sigma,\Lambda}.$$
 \end{Definition}
 
 \begin{Proposition}[Proposition 38 in \cite{BdL}]
 	Let $X\in\cX^1(M)$ and $U$ be a compact set whose singularities are all hyperbolic.
 	Let $\Lambda_{X,U}$ be the maximal invariant set of $X$ in $U$.
 	Then there exists a $C^1$-neighborhood $\cU$ of $X$ where the map $Y\in\cU\mapsto B(Y,\Lambda_{U,X})$ is upper semi-continuous.   
 \end{Proposition}
 
 \subsection{Renormalization cocycle associated to a hyperbolic singularity}
 Let $X\in\cX^1(M)$ and let us consider the subset of $\cG^1$ defined by
 $$\widetilde{M}_X=\overline{\big\{\RR X(x): x\in M\setminus\sing(X)\big\}}\cup\bigcup_{\sigma\in\sing(X)} \cG^1(\sigma).$$
 A real-valued \emph{cocycle} over $(\wh\varphi_t)_{t\in\RR}$ is a continuous function $H:\widetilde{M}_X\times\RR\to (0,+\infty)$ such that 
 $$H(L, t+s)=H(L,t)\cdot H(\wh\varphi_t(L),s)\textrm{ for any $t,s\in\RR$ and $L\in \widetilde{M}_X$}.$$
 We will write $H(L,t)=h^t(L)$ and $H=(h^t)_{t\in\RR}$.
 \begin{Definition}~\label{Def:renormalization-cocycle}
 	Let $X\in\cX^1(M)$ and let $\sigma$ be a hyperbolic singularity.
 	A cocycle $(h^t)_{t\in\RR}$ over the flow $(\wh\varphi_t)_{t\in\RR}$ is a \emph{renormalization cocycle at $\sigma$} if 
 	\begin{itemize}
 		\item there exist  a neighborhood $U_\sigma$ of $\sigma$ and $C>1$ such that
 		for any $x\in U_\sigma$, $L\in\cG^1(x)\cap \widetilde M_X$ and $t\in\RR$ satisfying $\varphi_t(x)\in U_\sigma$, one has
 		$$C^{-1}<\frac{h^t(L)}{\|D\varphi_t|_L\|}<C;$$
 		\item  for any small neighborhood $W$ of $\sigma$, there exists $C_W>1$ such that for any $x\in M\setminus W$, $L\in\cG^1(x)\cap \widetilde M_X$
 		and $t\in\RR$  satisfying $\varphi_t(x)\in M\setminus W$, one has
 		$C_W^{-1}<h^t(L)<C_W.$
 	\end{itemize}
 \end{Definition}
 At any hyperbolic singularity $\sigma$, there exists a renormalization cocycle, and it is unique up to multiplication by a cocycle bounded away from $0$ and $+\infty$
 (see Theorem 1 in~\cite{BdL}).
 
 The following property appears in the Corollary 63 of~\cite{BdL} and justifies the renormalization by the cocycle $(h^t)$.
 \begin{Proposition}~\label{Prop:lorenz-cocycle}
 	Let $X\in\cX^1(M)$, let $\sigma$ be a Lorenz-like singularity with splitting $T_\sigma M=E^{s}\oplus E^c\oplus E^{uu}$
 	and let $\PP E^{cs}_\sigma$ denote the projective space of $E^s\oplus E^c$, so that the extended linear Poincar\'e flow admits a dominated splitting $\cN^s\oplus \cN^u$
 	with $\dim(\cN^s)=\dim(E^{s})$ over $\PP E^{cs}_\sigma$.
 	
 	If $(h^t)$ is a renormalization cocycle at $\sigma$, then the cocycle $(h^t\cdot\wh\Psi_t|_{\cN^s})$ contracts uniformly.
 \end{Proposition}

 \subsection{Bonatti-da Luz's definition}
 We can now recall the definition introducted in~\cite{BdL}.
 
 \begin{Definition}~\label{Def:multi-singular-hyperbolic-bonatti-da-luz}
 	Let $X\in\cX^1(M)$. An invariant compact set $\Lambda$ is \emph{multi-singular hyperbolic} (in the sense of Bonatti-da Luz) if:
 	\begin{enumerate}[i.]
 		\item The singularities $\sigma\in \Lambda$ are hyperbolic. We fix a renormalization cocycle $(h^t_\sigma)$ at each $\sigma$.
 		\item The extended linear Poincar\'e flow admits a dominated splitting $\cN^s\oplus \cN^u$
 		over $B(X,\Lambda)$.
 		\item There exists a subset $S_+\subset \sing(X)\cap \Lambda$ such that the cocycle $(h_+^t\cdot\wh\Psi_t|_{\cN^s})$ is uniformly contracting,
 		where $h_+^t=\prod_{\sigma\in S_+} h_\sigma^t$.
 		\item There exists a subset $S_-\subset \sing(X)\cap \Lambda$ such that the cocyle $(h_-^t\cdot\wh\Psi_t|_{\cN^u})$ is uniformly expanding,
 		where $h_-^t=\prod_{\sigma\in S_-} h_\sigma^t$.
 	\end{enumerate}
 	One says that $X$ is multi-singular hyperbolic in a compact set $U$ if the maximal invariant set of $X$ in $U$
 	is multi-singular hyperbolic.
 \end{Definition}
 \begin{Remark}
 	Under the assumption that $W^s(\sigma)\cap\Lambda\setminus\{\sigma\}\neq\emptyset$ and $W^u(\sigma)\cap\Lambda\setminus\{\sigma\}\neq\emptyset$ for all singularities $\sigma\in\Lambda$, the set $S_+$ (resp. $S_{-}$) has to coincide with the set of  singularities whose stable dimension is  $\dim(\cN^s)+1$ (resp. $\dim(\cN^s)$) (see the proof of  Proposition~\ref{p.bonatti-da-luz-implies-our}).
 \end{Remark}
 We then compare  Definitions~\ref{Def:multi-singular-hyperbolic} and~\ref{Def:multi-singular-hyperbolic-bonatti-da-luz}.
 We  first show that the first implies the second.
 
 \begin{theoremalph}~\label{thm.our-imply-bonatti-da-luz}
 	Let $X\in\cX^1(M)$ and let $\Lambda$ be an invariant compact set. If $\Lambda$ satisfies Definition~\ref{Def:multi-singular-hyperbolic}, then it satisfies Definition~\ref{Def:multi-singular-hyperbolic-bonatti-da-luz}.
 \end{theoremalph}
 \proof
 Let $\cN^s\oplus \cN^u$ be the dominated splitting for the linear Poincar\'e flow over $\Lambda\setminus \Sing(X)$ as in Definition~\ref{Def:multi-singular-hyperbolic}.
 It extends to the closure in $\cG^1$, hence to $\overline{\{\RR X(x): x\in \Lambda\setminus\sing(X) \}}$. Let $i$ be its index.
 Let us consider a singularity $\sigma\in\Lambda$, and let us assume that it has the splitting $T_\sigma M=E^{ss}\oplus E^c\oplus E^u$, with
 $\dim(E^{ss})=i$ and $\dim(E^c)=1$, and   $W^{ss}(\sigma)\cap \Lambda\setminus \{\sigma\}=\emptyset$   (the other cases  can be addressed analogously).
 Then $\PP^c_{\sigma,\Lambda}$ is contained in the projective space
 associated to $E^c\oplus E^u$. By Proposition~\ref{Prop:lorenz-cocycle}, there exists a dominated splitting of index $i$ for the extended normal flow over $\PP^c_{\sigma,\Lambda}$.
 Hence the dominated splitting extends to $B(X,\Lambda)$.
 
 We now check the item (iii) of Definition~\ref{Def:multi-singular-hyperbolic-bonatti-da-luz}
 (the item (iv) is checked analogously). The set $S_+$ is the set of singularities in $\Lambda$ with a dominated splitting
 $T_\sigma M=E^{s}\oplus E^c\oplus E^{uu}$ and let $(h_+^t)$ be the associated cocycle.
 In order to prove that $(h_+^t\cdot\wh\Psi_t|_{\cN^s})$ is uniformly contracting, we have to prove that for any ergodic probability $\wh\mu$
 on $B(X,\Lambda)$, there exists $T>0$ such that
 \begin{equation}\label{e.contraction}
 \int \log h_+^T+\log \|\wh\Psi_T|_{\cN^s}\|\ud\wh\mu< 0.
 \end{equation}
 Proposition~\ref{Prop:lorenz-cocycle} proves that it is the case for the measures supported on the invariant sets
 $\PP^c_{\sigma,\Lambda}\subset \PP E^{cs}_\sigma$ associated to singularities $\sigma\in S_+$.
 For singularities $\sigma\in S_-$, $\PP^c_{\sigma,\Lambda}$ is contained in the projective space of $\PP E^{cu}_\sigma$,
 above which the cocycle $(\wh\Psi_t|_{\cN^s})$ is uniformly contracting; since $(h_+^t)$ is bounded, the property~\eqref{e.contraction}
 holds for measures supported $\PP^c_{\sigma,\Lambda}$ also in this case.
 
 It remains to consider ergodic measures $\wh \mu$ which projects on a regular measure $\nu$ on $\Lambda$.
 For each $T>0$ we have
 \begin{eqnarray}~\label{equ:non-hyperbolic}
 \int\log h_+^T+\log \|\wh\Psi_T|_{\cN^s}\|\ud\wh\mu=\int \log h_+^T+\log \|\Psi_T|_{\cN^s}\|\ud\nu.
 \end{eqnarray}
 By Theorem~\ref{thm.robustness-of-multi-singular-hyperbolicity}, $X$ satisfies the star property on a neighborhood of $\Lambda$.
 By the proof of Theorem 5.6 in \cite{SGW}, each regular ergodic measure $\nu$ supported on $\Lambda$ is accumulated by periodic measures
 $\delta_{\gamma_n}$ supported on periodic orbits $\gamma_n$ contained in a small neighborhood of $\Lambda$.
 Hence, there exists a sequence of periodic orbits $\gamma_n$ such that $\delta_{\gamma_n}$ tends to $\nu.$
 Notice that the singular domination over $\Lambda$ can be extended continuously to the maximal invariant set of $X$ in $U$ by Theorem~\ref{p.robust-singular-domination}, which implies 
 \begin{equation}\label{e.equ:non-hyperbolic2}
 \int \log h_+^T+\log \|\Psi_T|_{\cN^s}\|\ud\nu=\lim_{n\rightarrow\infty}\int \log h_+^T+\log \|\Psi_T|_{\cN^s}\|\ud\delta_{\gamma_n}.
 \end{equation}
 On $\gamma_n$, the cocycle $(h_+^t)$ is bounded away from $0$ and $+\infty$. 
 Birkhoff ergodic theorem gives
 \begin{equation}\label{e.equ:non-hyperbolic3}
 \int \log h_+^T \ud\delta_{\gamma_n}=\int\lim_{k\rightarrow\infty}\frac{1}{k}\sum_{i=0}^{k-1}\log h_+^T(\varphi_{iT}(p))\ud\delta_{\gamma_n}
 = \int\lim_{k\rightarrow\infty}\frac{1}{k}h_+^{kT}(p)\ud\delta_{\gamma_n}
 =0.
 \end{equation}
 The star property and Theorem~\ref{thm:fundamental-property} give $T>0 $ and $\eta>0$ so that
 \begin{equation}\label{e.equ:non-hyperbolic4}
 \int \log \|\Psi_T|_{\cN^s}\| \ud\delta_{\gamma_n}= \int\lim_{k\rightarrow\infty}\frac{1}{k}\log\prod_{i=0}^{k-1}\|\Psi_T(\varphi_{iT}(p))|_{\cN^s}\|\ud\delta_{\gamma_n}
 \leq -\eta.
 \end{equation}
 The Equations~\eqref{equ:non-hyperbolic}, \eqref{e.equ:non-hyperbolic2}, \eqref{e.equ:non-hyperbolic3} and \eqref{e.equ:non-hyperbolic4}
 together imply~\eqref{e.contraction} for the regular measure $\wh \mu$.
 \endproof
 
 Then we show the converse under a mild condition.
 \begin{theoremalph}~\label{thm.bonatti-da-luz-implies-our}
 	Let $X\in\cX^1(M)$ and $\Lambda$ be an invariant compact set whose singularities $\sigma$ satisfy  $W^s(\sigma)\cap\Lambda\setminus\{\sigma\}\neq\emptyset$ and $W^u(\sigma)\cap\Lambda\setminus\{\sigma\}\neq\emptyset$.
 	If $\Lambda$ satisfies Definition~\ref{Def:multi-singular-hyperbolic-bonatti-da-luz}, it also satisfies Definition~\ref{Def:multi-singular-hyperbolic}.
 \end{theoremalph}
 We need an auxiliary result before proving Theorem~\ref{thm.bonatti-da-luz-implies-our}.
 \begin{Proposition}~\label{p.bonatti-da-luz-implies-our}
 	Let $X\in\cX^1(M)$ and $\Lambda$ be an invariant compact set satisfying Definition~\ref{Def:multi-singular-hyperbolic-bonatti-da-luz}.
 	Then any singularity $\sigma\in \Lambda$ such that $W^s(\sigma)\cap\Lambda\setminus\{\sigma\}\neq\emptyset$ and $W^u(\sigma)\cap\Lambda\setminus\{\sigma\}\neq\emptyset$
 	is Lorenz-like.  Moreover if $\sigma$ has the splitting $T_\sigma M=E^{ss}\oplus E^{cs}\oplus E^u$,
 	then $W^{ss}(\sigma)\cap \Lambda=\{\sigma\}$.
 \end{Proposition}
 \proof
 Let $\cN^s\oplus \cN^u$ be the domination over $B(X,\Lambda)$ for the extended linear Poincar\'e flow  as in Definition~\ref{Def:multi-singular-hyperbolic-bonatti-da-luz}
 and let $\sigma\in U$ be a singularity as in the statement of the proposition.
 By assumption, the center space  $\PP^c(\sigma,U)$ contains lines  $L_s\subset E^s(\sigma)$ and $L_u\subset E^u(\sigma)$. Without loss of generality, one  assumes that $\dim(E^s(\sigma))>\dim(\cN^s)$. Applying  Proposition~\ref{p.strong-stable-outside} to the domination over $\PP^c(\sigma,U)$ for the extended linear Poincar\'e flow, there exists a dominated splitting $T_\sigma M=E^{ss}\oplus E^{cs}\oplus E^u$ for the tangent flow with $\dim(E^{ss})=\dim(\cN^s)$.
 
 We claim that $W^{ss}(\sigma)\cap\Lambda=\{\sigma\}$.
 If one assumes by contradiction that this does not hold,
 there exists a line $L^s\subset E^{ss}$ which belongs to $\PP^c_{\sigma,\Lambda}$, consequently, there exists
 a line in $E^{ss}\oplus E^u$ which is not contained in $E^{ss}\cup E^u$ and belongs to $\PP^c_{\sigma,\Lambda}\subset B(X,\Lambda)$.
 This contradicts the second item of Proposition~\ref{p.strong-stable-outside}.
 
 In particular $L_s\subset E^{cs}(\sigma)$.
 Up to changing the metric, one can assume that the splitting $T_\sigma M=E^{ss}\oplus E^{cs}\oplus E^u$ is orthogonal.
 Then $E^{cs}(\sigma)\subset \cN^u(L_u)$. In order to satisfy the item (iv) of Definition~\ref{Def:multi-singular-hyperbolic-bonatti-da-luz},
 the bundle $\cN^u$ over  $\PP^c_{\sigma,\Lambda}$ has to be renormalized by the cocycle $(h^t_\sigma)$,
 proving that $\sigma$ belongs to $S_-$.
 
 From the first item of Definition~\ref{Def:renormalization-cocycle} and item (iv) in Definition~\ref{Def:multi-singular-hyperbolic-bonatti-da-luz},
 $(\|D\varphi_t|_{L_s}\|\cdot \wh\Psi_{t}|_{\cN^u(L_s)})$ is uniformly expanding along the orbit of $L_s\in \PP^c_{\sigma,\Lambda}$.
 This implies that $E^{cs}=L_s$ is one-di\-men\-sional and that the Lyapunov exponent $\lambda^c$ along $E^{cs}$ and the smallest Lyapunov exponent $\lambda^u$ along $E^u$
 satisfy $\lambda^c+\lambda^u>0$. Hence $\sigma$ is Lorenz-like.
 \endproof
 
 \proof[Proof of Theorem~\ref{thm.bonatti-da-luz-implies-our}]
 Let $\cN^s\oplus \cN^u$ be the dominated splitting given by Definition~\ref{Def:multi-singular-hyperbolic-bonatti-da-luz} and let $i$ be its index.
 By  Proposition~\ref{p.bonatti-da-luz-implies-our}, each singularity in $\Lambda$ has a dominated splitting
 $T_\sigma M=E^{ss}\oplus E^c\oplus E^{uu}$ with $\dim(E^c)=1$, $\dim(E^{ss})=i$
 and either $W^{ss}(\sigma)\cap \Lambda=\{\sigma\}$ or $W^{uu}(\sigma)\cap \Lambda=\{\sigma\}$.
 This proves that $\Lambda$ admits a singular domination of index $i$.
 
 For each singularity $\sigma\in \Lambda$, let $(h^t_\sigma)$ be the renormalization cocycle at $\sigma$ and consider a small closed neighborhood $V_\sigma$ of $\sigma$ such that
 \begin{itemize}
 	\item the maximal invariant set of $(\varphi_t)_{t\in\RR}$ in $V_\sigma$ is $\sigma$;
 	\item  $V_\sigma$ is contained in the neighborhood of $\sigma$ given in the first item of   Definition~\ref{Def:renormalization-cocycle}.
 \end{itemize}
 Taking the neighborhoods $V_\sigma$ small enough, one can assume that they are pairwise disjoint
 and let $C(\sigma)>1$ be the constant associated to $V_\sigma$ by the  second item in Definition~\ref{Def:renormalization-cocycle} .   
 Take $$C=  \prod_{\sigma\in \sing(X)\cap \Lambda} C(\sigma).$$
 By Definition~\ref{Def:multi-singular-hyperbolic-bonatti-da-luz}, there exists $\eta>0$ such that for each $L\in B(X,\Lambda)$, one has for any $t$ large
 $$\|h_+^t(L)\cdot \wh\Psi_t|_{\cN^s(L)}\|<e^{-2\eta t}\textrm { and } \|h_-^{-t}\cdot \wh\Psi_{-t}|_{\cN^u(L)}\|<e^{-2\eta t}.$$
 Fix $T$ large such that $C\cdot e^{-\eta T}<1.$ Now, for any $x\in\Lambda$ and $t>T$ such that $x$ and $\varphi_t(x)$ are disjoint from $V:=\cup_{\sigma\in\sing(X)\cap U} V_\sigma$, denoting $L=\RR X(x)$, then one has
 $$
 \|\Psi_t|_{\cN^s(x)}\|
 <C\cdot \|h_+^t(L)\cdot \wh\Psi_t|_{\cN^s(L)}\|\leq e^{-\eta t}.
 $$
 Similarly, one can show $\|\Psi_{-t}|_{\cN^u(x)}\|<e^{-\eta t}.$
 \endproof

\noindent{\it Acknowledgments.}
We would like to thank Christian Bonatti, Lan Wen, Shaobo Gan and  Yi Shi  for useful comments and discussions. This paper was prepared during the stay of Jinhua Zhang at Universit\'e Paris-sud, and Jinhua Zhang would like to thank Universit\'e Paris-sud for hospitality.

\begin{tabular}{l l l}
	\emph{\normalsize Sylvain Crovisier}
	& \quad &
	\emph{\normalsize Adriana da Luz}
	\medskip\\
	
	\small Laboratoire de Math\'ematiques d'Orsay
	&& \small Beijing International Center for Mathematical Research\\
	\small CNRS - Universit\'e Paris-Sud
	&& \small Peking University\\
	\small Orsay 91405, France
	&& \small Peking, 100871, P.R. China\\
	\small \texttt{Sylvain.Crovisier@math.u-psud.fr}
	&& \small \texttt{adaluz@cmat.edu.uy}\\
	&& \small
	\bigskip\\
	\emph{\normalsize Dawei Yang}
	& \quad &
	\emph{\normalsize Jinhua Zhang}
	\medskip\\
	
	\small School of Mathematical Sciences
	&& \small School of Mathematical Sciences\\
\small Soochow University
&& \small Beihang University\\
\small Suzhou, 215006, P.R. China
&& \small Beijing 100191, P.R. China\\
\small \texttt{yangdw1981@gmail.com}
&&\small \texttt{zjh200889@gmail.com}\\
\small \texttt{yangdw@suda.edu.cn}
&&\small \texttt{jinhua$\_$zhang@buaa.edu.cn}
	
\end{tabular}

\end{document}